\documentclass[11pt,a4paper]{amsart}

\usepackage{amsmath,amsfonts,amscd,amssymb,amsthm,bm,amsrefs}
\usepackage{verbatim}
\usepackage{color}
\usepackage{hyperref}
\usepackage{a4wide}

\theoremstyle{plain}
\newtheorem{lemma}{Lemma}[section]
\newtheorem{theorem}[lemma]{Theorem}

\theoremstyle{definition}

\theoremstyle{remark}
\newtheorem{remark}[lemma]{Remark}

\newcommand{\R}{\mathbb{R}}
\newcommand{\C}{\mathbb{C}}
\newcommand{\Z}{\mathbb{Z}}
\newcommand{\D}{\mathbb{D}}
\newcommand{\N}{\mathbb{N}}

\newcommand{\wind}{\operatorname{wind}}

\newcommand{\jtil}{\widetilde{J}}
\newcommand{\util}{\widetilde{u}}
\newcommand{\vtil}{\widetilde{v}}
\newcommand{\wtil}{\widetilde{w}}
\newcommand{\CZ}{{\rm CZ}}
\newcommand{\jbar}{\bar{J}}
\renewcommand{\P}{\mathcal{P}}
\newcommand{\Mfast}{\mathcal{M}^{\rm fast}}
\newcommand{\ev}{{\rm ev}}
\newcommand{\W}{\overline{W}}

\begin{document}

\title[On the knot types of periodic Reeb orbits]{On the knot types of periodic Reeb orbits of dynamically convex contact forms}
\author[]{Umberto L. Hryniewicz}
\author[]{Pedro A. S. Salom\~ao}
\author[]{Richard Siefring}

\address{Umberto L.\ Hryniewicz, Lehrstuhl f\"ur Geometrie und Analysis, RWTH Aachen University, Pontdriesch 10-12, D-52062 Aachen, Germany}
\email{hryniewicz@mathga.rwth-aachen.de}

\address{Richard Siefring}

\address{Pedro A. S. Salom\~ao, Shenzhen International Center for Mathematics - SUSTech, 1088 Xueyuan Avenue,  Shenzhen, China}
\email{psalomao@sustech.edu.cn}

\date{\today}

\begin{abstract}
We exhibit transverse knot types on the standard contact $3$-sphere that cannot be realized as periodic Reeb orbits of a dynamically convex contact form.
In particular, such transverse knot types do not arise as closed characteristics of strictly convex energy levels on a four dimensional symplectic vector space.
\end{abstract}

\maketitle

\section{Introduction and main results}

Knotted structures in the dynamics of flows in dimension three have been studied since the XIXth century with the works of Lord Kelvin and Helmholtz. The task of understanding knots of periodic orbits is a topic that was visited by several authors during the XXth century. For example, in a foundational paper~\cite{BiWill} Birman and Williams prove that all periodic orbits in the Lorenz attractor are fibered knots. More recently, it was proved in~\cite{AF} that the same is true for every periodic orbit of the geodesic flow of a Riemannian two-sphere with an explicit pinching condition on the curvatures. Here we consider Reeb flows on closed contact $3$-manifolds and study the transverse knot types realized by their closed orbits. 

As observed by Etnyre and Ghrist~\cite{EG}, one may take the complexity of the transverse knots realized by the closed orbits as a rough measure of the topological complexity of the flow. In the following discussion consider, for simplicity, the case of the standard contact $3$-sphere $(S^3,\xi_0)$, where $S^3=\{(z_0,z_1)\in\C^2:|z_0|^2+|z_1|^2=1\}$ and $\xi_0 = TS^3\cap iTS^3$. One may ask how complex a Reeb flow on $(S^3,\xi_0)$ can be from this viewpoint. The result from~\cite{EG} asserts that there exist real-analytic examples where all transverse knot types are simultaneously realized. The main goal of this paper is to study concrete geometric properties that prevent this phenomenon from happening, and force restrictions on the transverse knot types that can arise as closed orbits. Specifically, we look at the property of dynamical convexity, an important concept introduced by Hofer, Wysocki and Zehnder (HWZ) in~\cite{convex}. The precise definition will be recalled below. Two rich sources of examples are the strictly convex energy levels in a symplectic vector space, see Theorem~3.4 in~\cite{convex}, and the unit cotangent bundles of Finsler two-spheres with reversibility $r\geq1$ and flag curvatures pinched by strictly more than $(r/(r+1))^2$, see~\cites{HP,HS_Cambridge}. Additionally, dynamical convexity has been observed in numerous celestial mechanics models~\cites{dPKSS2024,JK,LS2025,HLOSY2025}.

The following terminology is useful: a transverse knot in a contact $3$-manifold will be called a \textit{Hopf fiber} if it is unknotted and has self-linking number $-1$ with respect to some embedded spanning disk. The choice of term is motivated by the fact that fibers of the standard Hopf fibration are transverse unknots with self-linking number $-1$ in $(S^3,\xi_0)$. 

The \textit{standard Liouville form} in $\C^2$ is $\alpha_0 = (-i/4) \ \Sigma_{j=0,1} \ \bar z_j dz_j - z_jd\bar z_j$, $(z_0,z_1)\in\C^2$, and the \textit{standard symplectic form} is $\omega_0 = d\alpha_0$. By a star-shaped domain $\Omega \subset \C^2$ we mean a compact connected domain with a smooth connected boundary such that: (1) $t\Omega\subset\Omega \ \forall t\in[0,1]$ and $0\not\in\partial\Omega$, (2) $\partial\Omega$ is transverse to rays issuing from the origin. Note that $\alpha_0$ induces a contact form on $\partial\Omega$. The corresponding contact structure on $\partial\Omega$ is denoted by $\xi_0 = \ker\alpha_0 \cap T\partial\Omega$. Dynamical convexity of $\Omega$ means, by definition, that the Conley-Zehnder index of every periodic Reeb orbit is at least $3$ when computed in a global $\omega_0$-symplectic trivialization of~$\xi_0$.

If $(V,\xi)$ is a contact manifold and $V$ is diffeomorphic to $S^3$, then we say that a knot $K \subset S^3$ admits a transverse representative in $(V,\xi)$ if there exists a diffeomorphism $\varphi:S^3\to V$ such that $\varphi(K)$ is transverse to $\xi$. Our main result reads as follows.

\begin{theorem}
\label{main1}
Let $\Omega$ be a dynamically convex star-shaped domain. A closed characteristic on $\partial\Omega$ is a Hopf fiber if, and only if, it is symplectically slice in $(\Omega,\omega_0)$, i.e. bounds an embedded symplectic disk in $(\Omega,\omega_0)$. In particular, the knots $8_{20}$, $9_{46}$, $10_{140}$ and $10_{155}$ in $S^3$ (Rolfsen's table~\cite{rolfsen}) admit transverse representatives in $(\partial \Omega,\xi_0)$ that cannot be realized as closed characteristics.
\end{theorem}

\begin{remark}
The Hopf flow is ``simplest'' from the viewpoint of knot complexity: only Hopf fibers appear. It is, of course, far from true that a Reeb flow on $(S^3,\xi_0)$ all of whose closed orbits are Hopf fibers is as simple as the Hopf flow: work of Anosov and Katok~\cite{AnosovKatok} provides transitive examples with exactly two closed orbits in such a way that the union of the closed orbits is transversely isotopic to a pair of Hopf fibers as seen in the Hopf flow; we call these \textit{Hopf links}. Work of Albers, Geiges and Zehmisch~\cite{AGZ_Katok} explores the Anosov-Katok method in the context of Reeb flows in three dimensions. See also the appendix of the paper~\cite{ABE} by Abbondandolo, Benedetti and Edtmair where this construction in the transitive case is reformulated in modern symplectic language. Moreover, a special case of the result in~\cite{CGHHL1} implies that the union of the closed orbits of a Reeb flow on $(S^3,\xi_0)$ with exactly two closed orbits is transversely isotopic to a Hopf link. 
\end{remark}

\begin{remark}
One may ask, as in~\cite{EG}, what transverse knots/links must be realized as periodic orbits for all Reeb flows on $(S^3,\xi_0)$. Work of HWZ~\cite{HWZ_unknotted} implies that the Hopf fiber is always realized. The lack of counter-examples leads us to conjecture that \textit{every Reeb flow on $(S^3,\xi_0)$ possesses a pair of closed orbits transversely isotopic to a Hopf link.} In other words, a Hopf link should be the minimal basic structure that is always present. Note that the lift of the geodesic flow of a reversible Finsler two-sphere, from the unit sphere bundle to $(S^3,\xi_0)$, is a Reeb flow that realizes a Hopf link: simply lift the unit vectors tangent to an embedded closed geodesic. Hopf links follow hand-in-hand the historical development of Symplectic Dynamics since Poincar\'e's studies of the planar circular restricted $3$-body problem: the retrograde and direct orbits lift to a Hopf link after Levi-Civita regularization, both in a low energy regime or in a subcritical low/high mass-ratio regime. These (lifted) orbits span the annulus-like global surface of section that led Poincar\'e to state his celebrated \textit{last geometric theorem}, later known as the Poincar\'e-Birkhoff theorem, which in turn led Arnold to state his conjecture on the minimal number of fixed points of Hamiltonian diffeomorphisms. Arnold's conjecture paved the way to Floer theory, we refer to the book by Hofer and Roberts~\cite{Helmut_book} for a detailed historical account. Strong dynamical consequences of the existence of a Hopf link of periodic orbits were obtained in~\cite{HMS}, for instance, under a non-resonance assumption, the growth of closed orbits with respect to period is at least quadratic.
\end{remark}

There are results available in the literature about transverse knot types for dynamically convex Reeb flows on $(S^3,\xi_0)$. It follows from results in~\cite{convex} that every such flow has a Hopf link made of periodic orbits, and the results in~\cite{hryn_jsg},~\cite{HSW} imply that each component spans a disk-like global surface of section (GSS), and the link spans an annulus-like GSS. In~\cite{HSS}, it is shown that if $L$ is a link made of periodic orbits of a dynamically convex Reeb flow on $(S^3,\xi_0)$, such that every component of $L$ is a Hopf fiber, then $L$ is a fibered link and spans a GSS. It also follows from the results in~\cite{HSS} combined with results by Abbondandolo, Edtmair, and Kang in~\cite{AEK} that, in the nondegenerate case, the periodic Reeb orbit of minimal action spans a GSS, although it is not known if this GSS is a disk. This latter statement is also valid without the nondegeneracy assumption.

Theorem~\ref{main1} will be proved as a consequence of Theorem~\ref{main2} which is a dynamical characterization of smooth compact star-shaped domains in $(\C^2,\omega_0)$.

Consider a compact symplectic manifold $(W,\omega)$ with contact-type boundary $(V,\lambda)$. This means that there exists a Liouville vector field $Y$ defined on a neighborhood of $V=\partial W$ in $W$ that is transverse to $V$, and $\lambda$ is the pull-back of $i_Y\omega$ by the inclusion map $V\hookrightarrow W$. It follows that $\lambda$ is a contact form on $V$. The boundary gets decomposed as $V=V^+\cup V^-$, where $Y$ points outwards along $V^+$ and inwards along $V^-$. One calls $V^+$ the \textit{convex} part of the boundary and $V^-$ the \textit{concave} part of the boundary. The corresponding contact structure on $V$ is denoted $\xi=\ker\lambda$. 

A transverse link in $(V^+,\xi)$ is called \textit{symplectically null-homologous} if it bounds an embedded symplectic surface in $(W,\omega)$. By a symplectic surface we mean one where~$\omega$ restricts to an area form. If the surface is a disk then its boundary is called a \textit{symplectically slice} knot. 

Let $P$ be a periodic Reeb orbit on $V^+$ with primitive period $T>0$, and $D \subset W$ be a symplectic slicing disk for $P$. Let $V_0^+ \subset V^+$ be the connected component that contains~$P$. A pair $(P',D')$, consisting of a periodic Reeb orbit $P' \subset V_0^+$ of (not necessarily primitive) period $T'>0$ and a capping disk $D'$ for $P'$ in $W$, is said to be \textit{short and unlinked relative to $(P,D)$} if $T'<T$, and if $D$ and $D'$ are homotopic through capping disks to a pair of disjoint capping disks. We say that $(V^+_0,\lambda)$ is \textit{dynamically convex in $W$ relative to $(P,D)$} if:
\begin{itemize}
\item[(i)] The Conley-Zehnder index of $P$ relative to $D$ is at least $3$.
\item[(ii)] The Conley-Zehnder index of $P'$ relative to $D'$ is at least $3$ for every $(P',D')$ that is short and unlinked relative to $(P,D)$.
\end{itemize}
See \S~\ref{ssec_CZ} for a precise definition  of the Conley-Zehnder index relative to a capping disk. We encode our analysis in the form of a symplectic characterization result for star-shaped domains in $(\C^2,\omega_0)$, see also~\cite{GZ}.

\begin{theorem}
\label{main2}
Let $(W,\omega)$ be a connected compact symplectic manifold with contact-type boundary $(V,\lambda)$. Suppose that $(W,\omega)$ is symplectically aspherical, i.e. $\int_{S^2}f^*\omega=0$ for every smooth map $f:S^2\to W$. Assume that some connected component $V_0^+ \subset V^+$ has a simple periodic orbit $P$ with a symplectic slicing disk $D$ such that $(V_0^+,\lambda)$ is dynamically convex in $W$ relative to $(P,D)$. Then either there is a periodic Reeb orbit on $V^-$ with period less than $\int_D\omega$, or all of the following hold:
\begin{itemize}
\item[(a)] $(V^+,\ker\lambda)$ is contactomorphic to $(S^3,\xi_0)$ and $V^- = \emptyset$.
\item[(b)] The orbit $P$ is unknotted and has self-linking number $-1$, i.e. it is a Hopf fiber.
\item[(c)] $(W,\omega)$ is symplectomorphic to a star-shaped region in $(\C^2,\omega_0)$.
\end{itemize}
\end{theorem}

\begin{remark}
Theorem~\ref{main2} recovers a celebrated result due to McDuff~\cite{McDuff_boundaries}*{Theorem~1.2}.
\end{remark}

\begin{remark}
If $(W,\omega)$ is an exact symplectic cobordism, then the first alternative in the conclusions of Theorem~\ref{main2} can be improved to provide a contractible periodic Reeb orbit in~$V_-$.
\end{remark}

\begin{remark}
If the symplectic asphericity assumption on $(W,\omega)$ is dropped, then one can still conclude that $(W,\omega)$ is symplectomorphic to a star-shaped region in $(\C^2,\omega_0)$ up to symplectic blow up at finitely many points. We do not implement this here since the main point of this work is to study knot types of periodic Reeb orbits.
\end{remark}

\begin{proof}
[Proof of Theorem~\ref{main1} using Theorem~\ref{main2}]
The statement that a closed characteristic in $\partial\Omega$ is symplectically slice in $(\Omega,\omega_0)$ if, and only if, it is a Hopf fiber, follows directly from Theorem~\ref{main2} applied to $(W,\omega)=(\Omega,\omega_0)$.

A source of symplectically null-homologous transverse links are the transverse intersections of a complex curve in $\C^2$ with the round $3$-sphere. Using results of Micallef and White~\cite{MW} we can perturb the complex structure to an almost complex structure, and the curve to a pseudo-holomorphic curve with only transverse double points. These can be swapped for genus using a standard trick, and we thus get an embedded symplectic surface spanning the link. Such links are examples of what are known as \textit{transverse $\C$-links}. The general definition allows for transverse intersections of a complex curve with a strictly pseudo-convex $3$-sphere bounding a Stein $4$-ball in $\C^2$, see Definition~4.74 in~\cite{R2}. But according to Proposition~4.75 in~\cite{R2}, it was observed by Boileau and Orevkov in~\cite{BO} that results of Eliashberg~\cite{E} imply that all transverse $\C$-links can be realized as intersections of complex curves with the round $3$-sphere. Moreover, Rudolph proved in~\cite{R1} that every so-called quasipositive link is smoothly isotopic to a transverse $\C$-link, and therefore smoothly isotopic to a symplectically null-homologous link. The converse is proved in~\cite{BO}, namely, every transverse $\C$-link is quasipositive or, more generally, every symplectically null-homologous link on the boundary of a symplectic $4$-ball is quasipositive.

The upshot is that a smooth link in $S^3$ is smoothly isotopic to a symplectically null-homologous transverse link in $(S^3,\xi_0)$ if, and only if, the link is quasipositive. One is then led to find symplectically slice knots by finding quasipositive knots with slice genus $0$. Since it is conceivable that a knot could be spanned in the $4$-ball by a symplectic surface with genus bigger than the slice genus, we need one more ingredient: a result by Gadgil and Kulkarni~\cite{GK} asserting that the symplectic surface minimizes genus on its relative homology class. Examples of quasipositive slice knots can be spotted in Rolfsen's knot table~\cite{rolfsen}: $8_{20}$, $9_{46}$, $10_{140}$ and $10_{155}$.
\end{proof}

\medskip

\noindent \textbf{Acknowledgments.} We would like to warmly thank Shah Faisal for spotting a mistake in the statement Theorem~\ref{main2} of a previous version of this paper: the periodic orbit claimed to exist in $V_-$ is not necessarily contractible. U.H. acknowledges the support from the RWTH Aachen, the SFB/TRR 191 ‘Symplectic Structures in Geometry, Algebra and Dynamics’, funded by the DFG (Projektnummer 281071066--TRR 191), and the FAPERJ Grant `Pesquisador Visitante' 203.664/2025. PS was partially supported by the National Natural Science Foundation of China (grant number W2431007). PS thanks the support of the Shenzhen International Center for Mathematics - SUSTech. RS gratefully acknowledges financial support from the DFG grant BR 5251/1-1, and from the SFB/TRR 191 “Symplectic Structures in Geometry, Algebra and Dynamics” funded by the DFG (Projektnummer 281071066--TRR 191) during the completion of this project.

\section{Proof of Theorem~\ref{main2}}\label{sec_thm_main1}

\subsection{Symplectic manifolds with contact-type boundary}
\label{ssec_symp_mfds}

Let $(W,\omega)$ be a compact symplectic $4$-manifold. Assume that there exists a Liouville vector field $Y$ defined on a neighborhood of $V = \partial W$ transverse to $V$. The primitive $\alpha = i_Y\omega$ of $\omega$ near $V$ pulls back to a contact form $\lambda$ on $V$ under the inclusion map $V\hookrightarrow W$. Then $V$ splits into $V = V^+ \sqcup V^-$, where $Y$ points outward along $V^+$ and points inward along $V^-$. The associated contact structure is denoted by $\xi = \ker \lambda \subset TV$. We call $(W,\omega)$ a \textit{symplectic cobordism} that is \textit{convex} at $(V^+,\lambda)$ and \textit{concave} at $(V^-,\lambda)$.

Let $\varphi^t_Y$ denote the (local) flow of $Y$, and with $\varepsilon>0$ small enough consider diffeomorphisms
\begin{equation}
\Phi^+ : (-\varepsilon,0] \times V^+ \to \mathcal{U}^+ \qquad \Phi^- : [0,\varepsilon) \times V^- \to \mathcal{U}^- \qquad \Phi^\pm(a,p) = \varphi^a_Y(p)
\end{equation}
onto neighborhoods $\mathcal{U}^\pm$ of $V^\pm$ in $W$. One can show that $(\Phi^\pm)^*\alpha = e^a\lambda$, where here we abuse notation and write $\lambda$ for the pull-back of $\lambda$ under the projection $\R \times V \to V$. The symplectic form gets represented by $$ \Phi^*\omega = d(e^a\lambda) = e^a ( da \wedge \lambda + d\lambda ) \, . $$ The \textit{symplectic completion} of $(W,\omega)$ is the symplectic manifold $(\overline{W},\overline{\omega})$ where $\overline{W}$ is defined as 
\begin{equation}
\label{completion}
\overline{W} = \left( W \sqcup (-\varepsilon,+\infty) \times V^+ \sqcup (-\infty,\varepsilon) \times V^- \right) / \sim
\end{equation}
where points are identified according to
\begin{equation*}
\Phi^+(a,p) \in \mathcal{U}^+ \sim (a,p) \in (-\varepsilon,0] \times V^+ \qquad \Phi^-(a,p) \in \mathcal{U}^- \sim (a,p) \in [0,\varepsilon) \times V^-
\end{equation*}
and the symplectic form $\overline{\omega}$ is defined as $$ \overline{\omega} = \omega  \ \text{on} \ W, \qquad\qquad \overline{\omega} = d(e^a\lambda) \ \text{on} \ (-\varepsilon,+\infty) \times V^+ \ \text{and on} \ (-\infty,\varepsilon) \times V^- \, . $$ Later it will be needed to consider the compactification
\begin{equation}
\overline{W}_\infty \ = \ \overline{W} \ \sqcup \ \{+\infty\} \times V^+ \ \sqcup \ \{-\infty\} \times V^-
\end{equation}
where the end $[0,+\infty)\times V^+$ is compactified to $[0,+\infty]\times V^+$ and the end $(-\infty,0]\times V^-$ is compactified to $[-\infty,0]\times V^-$.

Finally we construct symplectic forms $\overline{\omega}_b$ on $\overline{W}$, parametrized by $b>0$. These will be used later to define areas of pseudo-holomorphic curves in $\overline{W}$. For each $b>0$ consider a function $\phi_b^+:[0,+\infty) \to [1,+\infty)$ such that $(\phi_b^+)'>0$, $\phi_b^+$ agrees with $e^a$ near $0$, and $\phi_b^+(a) \to e^b$ as $a\to+\infty$. Define $\phi_b^-:(-\infty,0] \to (0,1]$ by $\phi^-_b(a)=(\phi^+_b(-a))^{-1}$. It follows that $\phi^-_b$ agrees with $e^a$ near $0$, $(\phi_b^-)'>0$, $\phi_b^-(a) \to e^{-b}$ as $a\to-\infty$. Define $\overline{\omega}_b$ by
\begin{equation}\label{symp_area_b}
\overline{\omega}_b = \begin{cases} \omega & \text{on} \ W \\ d(\phi_b^+\lambda) & \text{on} \ [0,+\infty)\times V^+ \\ d(\phi_b^-\lambda) & \text{on} \ (-\infty,0] \times V^- \end{cases} \, .
\end{equation}
according to~\eqref{completion}.

\subsection{Periodic Reeb orbits}

The Reeb vector field $X_\lambda$ of $\lambda$ is uniquely determined by $i_{X_\lambda}d\lambda=0$, $i_{X_\lambda}\lambda=1$. Let us fix a marked point on every periodic trajectory of the flow of $X_\lambda$. This flow is called the Reeb flow and is denoted here by $\phi^t$. By a periodic Reeb orbit we mean a pair $P=(x,T)$ where $x:\R\to V$ is a periodic Reeb trajectory such that $x(0)$ is the marked point, and $T>0$ is a period, not necessarily the primitive one. The set of periodic orbits will be denoted by $\P$. If $T_0>0$ is the primitive period of $x$ then $k = T/T_0 \in \N$ is called the \textit{multiplicity} of~$P$. The contact form $\lambda$ is said to be \textit{nondegenerate} if $d\phi^T|_{x(0)} : \xi_{x(0)} \to \xi_{x(0)}$ does not have eigenvalue $1$ for all $P=(x,T) \in \P$. 

We assume in the rest of this paper that $\lambda$ is nondegenerate since it suffices to prove Theorem~\ref{main2} in this case.

\subsection{Asymptotic operators}

Associated to any periodic Reeb orbit $P = (x,T)$ and any $d\lambda$-compatible complex structure $J:\xi\to\xi$ there is an unbounded operator on $L^2(x(T\cdot)^*\xi)$
\begin{equation*}
\eta \mapsto J(-\nabla_t\eta+T\nabla_\eta X_\lambda)
\end{equation*}
where $\nabla$ is a symmetric connection on $TV$ and $\nabla_t$ denotes the associated covariant derivative along the loop $t \in \R/\Z \mapsto x(Tt)$. This is called the \textit{asymptotic operator}, which turns out not to depend on the choice of $\nabla$. It is self-adjoint when $L^2(x(T\cdot)^*\xi)$ is equipped with the inner-product
\begin{equation*}
(\eta,\zeta) \mapsto \int_{\R/\Z} d\lambda(x(Tt))(\eta(t),J(x(Tt))\zeta(t)) \ dt
\end{equation*}
Its spectrum is discrete, consists of eigenvalues whose geometric and algebraic multiplicities coincide, accumulate at $\pm\infty$. Since $\lambda$ is nondegenerate, $0$ is not an eigenvalue of the asymptotic operators associated to all pairs $(P,J)$.

\subsection{Conley-Zehnder indices}\label{ssec_CZ}

The eigenvectors of the asymptotic operator associated to any $(P=(x,T),J)$ are nowhere vanishing sections of $x(T\cdot)^*\xi$ since these solve linear ODEs. Hence they have well-defined winding numbers with respect to a $d\lambda$-symplectic trivialization $\sigma$ of $x(T\cdot)^*\xi$. The winding number is independent of the choice of eigenvector of a given eigenvalue. This allows us to talk about the winding number $\wind_\sigma(\nu)$ of an eigenvalue $\nu$. For every $k\in\Z$ there are precisely two eigenvalues (counted with multiplicity) satisfying $\wind_\sigma = k$ and, moreover, $\nu_1\leq\nu_2 \Rightarrow \wind_\sigma(\nu_1)\leq\wind_\sigma(\nu_2)$. Given any $\delta \in \R$ we set
\begin{equation*}
\begin{aligned}
\alpha_{\sigma}^{<\delta}(P) &= \max \ \{ \wind_\sigma(\nu) \mid \nu \ \text{eigenvalue}, \ \nu<\delta \} \\
\alpha_{\sigma}^{\geq\delta}(P) &= \min \ \{ \wind_\sigma(\nu) \mid \nu \ \text{eigenvalue}, \ \nu\geq\delta \} \\
p_{\sigma,\delta}(P) &= \alpha_{\sigma}^{\geq\delta}(P) - \alpha_{\sigma}^{<\delta}(P)\end{aligned}
\end{equation*}
Finally we consider the constrained Conley-Zehnder index
\begin{equation}
\CZ_\sigma^\delta(P) = 2 \alpha_{\sigma}^{<\delta}(P) + p_{\sigma,\delta}(P)
\end{equation}
Moreover, two trivializations have a relative winding number and the associated Conley-Zehnder indices differ by twice the relative winding number.

If $D_0$ is a capping disk for the periodic Reeb orbit $P_0=(x_0,T_0)$ in $V^+$, then there is a unique (up to homotopy) $d\lambda$-symplectic trivializing frame $\sigma$ of $x_0(T_0\cdot)^*\xi$ that can be completed with $\{Y,X_\lambda\}$ on $\partial D_0$ to a global $\omega$-symplectic trivializing frame of $TW$ along $D_0$. The \textit{Conley-Zehnder index of $P_0$ relative to $D_0$} is defined to be $\CZ_\sigma^0(P_0)$.

\subsection{Pseudo-holomorphic curves}

Choose a $d\lambda$-compatible complex structure $J$ on~$\xi$. Following Hofer~\cite{93}, we define an almost complex structure $\jtil$ on $\R\times V$ by
\begin{equation}\label{R_inv_J}
\jtil : \partial_a \mapsto X_\lambda \qquad\qquad \jtil|_\xi = J
\end{equation}
where $X_\lambda$ and $\xi$ are seen as $\R$-invariant objects in $\R\times V$. Then $\jtil$ is $\R$-invariant and compatible with any symplectic form $d(\phi(a)\lambda)$ where $\phi,\phi'>0$. In particular, it is compatible with $d(e^a\lambda)$. Consider a closed Riemann surface $(S,j)$, a finite set $\Gamma \subset S$ and a $\jtil$-holomorphic map $\util = (a,u) : (S\setminus\Gamma,j) \to (\R\times V,\jtil)$ satisfying a finite-energy condition $$ 0< E(\util) = \sup_{\phi} \int_{S\setminus\Gamma} \util^*d(\phi\lambda) < \infty $$ where the supremum is taken over the set of $\phi:\R \to [0,1]$ satisfying $\phi'\geq0$. The number $E(\util)$ is called the \textit{Hofer energy}. The points of $\Gamma$ are called \textit{punctures}. A puncture $z\in\Gamma$ is called \textit{positive} or \textit{negative} if $a(w)\to+\infty$ or $a(w)\to-\infty$ when $w\to z$, respectively. It is called \textit{removable} if $\limsup |a(w)| < \infty$ when $w\to z$. It turns out that every puncture is positive, negative or removable, and that $\util$ can be smoothly extended across a removable puncture; see~\cite{93}.

\begin{remark}[Holomorphic polar coordinates]
Let $z\in\Gamma$ and let $K$ be a conformal disk centred at $z$, i.e. there is a biholomorphism $\varphi : (K,j,z) \to (\D,i,0)$. Then $K\setminus\{z\}$ admits positive holomorphic polar coordinates $(s,t) \in [0,+\infty) \times \R/\Z$ defined by $(s,t) \simeq \varphi^{-1}(e^{-2\pi(s+it)})$, and negative holomorphic polar coordinates $(s,t) \in (-\infty,0] \times \R/\Z$ defined by $(s,t) \simeq \varphi^{-1}(e^{2\pi(s+it)})$.
\end{remark}

Under the standing assumption that $\lambda$ is nondegenerate we have:

\begin{theorem}[HWZ~\cite{props1}]
Suppose that $z\in\Gamma$ is a nonremovable puncture, and $(s,t)$ are positive holomorphic polar coordinates at $z$. There exists $P=(x,T) \in \P$ and $d\in\R$ such that $u(s,t) \to x(\epsilon Tt+d)$ in $C^\infty(\R/\Z,V)$ as $s\to+\infty$, where $\epsilon = \pm1$ is the sign of the puncture.
\end{theorem}

The orbit $P$ is called the \textit{asymptotic limit} of $\util$ at $z$. Results by HWZ~\cite{props1}, further refined in~\cite{sie_CPAM}, explain that the asymptotic behavior of a finite-energy surface at a nonremovable puncture can be described in terms of asymptotic operators. To explain this point we need to introduce some notation. Let $P=(x,T) \in \P$ and $T_0>0$ be the primitive period of $x$. Denote the coordinates on $\R/\Z\times \D$ by $(\vartheta,z=x_1+ix_2)$, and set $\alpha_0 = d\vartheta+x_1dx_2$. A Martinet tube for $P$ is a diffeomorphism $$ \Psi:N\to\R/\Z\times \D $$ defined on a smooth compact neighborhood $N$ of $x(\R)$ such that: 
\begin{itemize}
\item[(MT1)] $\Psi(x(T_0\vartheta)) = (\vartheta,0)$ for all $\vartheta \in \R/\Z$.
\item[(MT2)] $\lambda|_N = \Psi^*(f\alpha_0)$, where $f : \R/\Z\times \D \to (0,+\infty)$ is smooth and satisfies $f(\vartheta,0) = T_0$, $df(\vartheta,0) = 0$ for all $\vartheta \in \R/\Z$.
\end{itemize}

Let $\util=(a,u):(S,\Gamma,j) \to (\R\times V,\jtil)$ be a finite-energy surface, and $z\in\Gamma$ be a nonremovable puncture where $\util$ is asymptotic to $P=(x,T)$. Let $(s,t)$ be positive or negative holomorphic polar coordinates at $z$ according to whether $z$ is a positive or negative puncture. We may assume, without loss of generality, that $u(s,0) \to x(0)$ as $|s|\to+\infty$. Fix any Martinet tube $\Psi:N\to\R/\Z\times \D$ for $P$, let $s_0>0$ be such that $|s|\geq s_0 \Rightarrow u(s,t) \in N$, and denote $(\vartheta(s,t),z(s,t)) = \Psi(u(s,t))$ when $|s| \geq s_0$.

\begin{theorem}[HWZ~\cite{props1}, Siefring~\cite{sie_CPAM}]\label{thm_asymptotics}
Let $\epsilon=\pm 1$ be the sign of the puncture~$z$ and $k\in\N$ be the multiplicity of $P$. Either $z(s,t)$ vanishes identically, or there is $r>0$ and an eigenvalue $\nu$ of the asymptotic operator associated to $(P,J)$ such that the following holds. There exists $c\in\R$ such that 
\begin{equation*}
\lim_{|s| \to +\infty} e^{r|s|}| D^\beta[a(s,t) - Ts - c] |= 0 \qquad\quad \forall D^\beta = \partial^{\beta_1}_s\partial^{\beta_2}_t
\end{equation*}
and if $\tilde\vartheta$ is a lift of $\vartheta$ then
\begin{equation*}
\lim_{|s| \to +\infty} e^{r|s|} | D^\beta[\tilde\vartheta(s,t) - kt] | = 0 \qquad \forall D^\beta = \partial^{\beta_1}_s\partial^{\beta_2}_t
\end{equation*}
The eigenvalue $\nu$ satisfies $\epsilon\nu<0$, and there exists an eigenvector $e(t)$ of $\nu$ and a $\C$-valued smooth function $R(s,t)$ defined for $|s|\geq s_0$ such that
\begin{equation*}
z(s,t) = e^{\nu s} \left( \hat e(t) + R(s,t) \right) \qquad \lim_{|s| \to +\infty} e^{r|s|}  |D^\beta R(s,t)| = 0 \qquad \forall D^\beta = \partial^{\beta_1}_s\partial^{\beta_2}_t
\end{equation*}
where $(0,\hat e(t)) = d\Psi|_{x(Tt)} \cdot e(t)$.
\end{theorem}

\begin{remark}\label{rmk_asymp_behaviour}
If $z(s,t)$ does not vanish identically then we say that $\util$ has a nontrivial asymptotic behavior at $z$. In this case we call $\nu$ and $e(t)$ the associated asymptotic eigenvalue and eigenvector, respectively. Otherwise $\util$ is said to have trivial asymptotic behavior at $z$. 
\end{remark}

Consider almost complex structures $\jbar$ on $\overline{W}$ that are $\omega$-compatible on $W$, and agree with an $\R$-invariant $\jtil$ as in~\eqref{R_inv_J} on $[0,+\infty)\times V^+ \cup (-\infty,0] \times V^-$. It follows that $\jbar$ is $\overline{\omega}$-compatible. As in~\cite{93},~\cite{sft_comp} we look at a closed Riemann surface $(S,j)$, a finite set $\Gamma \subset S$ and a $\jbar$-holomorphic map $\util:(S\setminus\Gamma,j) \to (\overline{W},\jbar)$ satisfying a finite-energy condition 
$$ 
\begin{aligned} 
0< E(\util) = \int_{\util^{-1}(W)} \util^*\omega \ &+ \ \sup_{\phi} \int_{\util^{-1}([0,+\infty)\times V^+)} \util^*d(\phi\lambda) \\ & \qquad\qquad + \int_{\util^{-1}((-\infty,0] \times V^-)} \util^*d(\phi\lambda) < \infty 
\end{aligned}
$$ where the supremum is taken over the set of smooth $\phi : \R \to [0,1]$ satisfying $\phi'\geq0$. The finite-energy condition implies that punctures are either removable or they behave like the nonremovable punctures explained before. If $(s,t)$ are positive holomorphic polar coordinates at a positive puncture $z\in\Gamma$ then $\util(s,t) \in [0,+\infty)\times V^+$ when $s\gg1$, and if one writes $\util=(a,u)$ for $s\gg1$ then $a(s,t) \to+\infty$ as $s\to+\infty$. There is a similar behavior at negative punctures, we leave details to the reader. Moreover, if $\lambda$ is nondegenerate then there is an asymptotic limit $P$ at a nonremovable puncture, and all the conclusions of Theorem~\ref{thm_asymptotics} hold. In particular, we can talk about nontrivial \textit{versus} trivial asymptotic behavior and asymptotic eigenvalues as in Remark~\ref{rmk_asymp_behaviour}.

\subsection{Fast planes and \texorpdfstring{$\wind_\infty$}{wind-infinity}}
\label{ssec_fast_planes}

Denote by $(\overline{\C} = \C \cup \{\infty\},i)$ the Riemann sphere, and consider a finite-energy plane $\util:(\C,i) \to (\overline{W},\overline{\omega})$ with a positive puncture~at $\infty$. Then $$ \util(\C\setminus B_R(0)) \subset [0,+\infty)\times V^+ $$ when $R>0$ is large enough. Hence, we can write $\util=(a,u)$ in components on $\C\setminus B_R(0)$ and get an inclusion of vector bundles
\begin{equation*}
(u|_{\C\setminus B_R(0)})^*\xi \subset (\util|_{\C\setminus B_R(0)})^*T\overline{W}
\end{equation*}
We can choose a trivializing $d\lambda$-symplectic frame $\{e_1,e_2\}$ of $(u|_{\C\setminus B_R(0)})^*\xi$ such that the frame $\{\partial_a,X_\lambda,e_1,e_2\}$ extends to a trivializing $\overline{\omega}$-symplectic frame of $\util^*T\overline{W}$. This frame is unique up to homotopy. Let the asymptotic limit of $\util$ at $\infty$ be $P = (x,T)$. If $|z|\gg1$ then we can write $\util(z) = (a(z),u(z)) \in [0,+\infty)\times V^+$ in components and assume, without loss of generality, that $u(R) \to x(0)$ as $R\to+\infty$. Using the asymptotic behavior we may further assume, up to homotopy, that the frame $\{e_1,e_2\}$ along the loop $u(Re^{i2\pi t})$ converges to a $d\lambda$-symplectic frame $\sigma$ of $x(T\cdot)^*\xi$ as $R\to+\infty$. The resulting frame $\sigma$ is determined by $\util$ up to homotopy, and we shall say that it is \textit{aligned with $\util$}. We denote
\begin{equation}
\CZ(\util) = \CZ^0_\sigma(P)
\end{equation}
where $\sigma$ is aligned with $\util$.

To define $\wind_\infty(\util)$ we follow~\cite{props2}. If $\util$ has nontrivial asymptotic behavior at~$\infty$, $e(t)$ is the asymptotic eigenvector given by Theorem~\ref{thm_asymptotics}, and $\hat e(t) \in \C^*$ is determined by representing $e(t)$ in a frame aligned with $\util$, then we set
\begin{equation*}
\wind_\infty(\util) = \frac{\theta(1)-\theta(0)}{2\pi} \in \Z
\end{equation*}
where $\theta(t)$ is a continuous choice of argument of $\hat e(t)$. If $\util$ has trivial asymptotic behavior at $\infty$ then we set $\wind_\infty(\util) = -\infty$.

The plane $\util$ is said to be \textit{fast} if $\wind_\infty(\util)\leq1$. Fast planes in symplectizations were introduced in~\cite{fast}. Theorem~\ref{thm_asymptotics} implies that if $\util$ has nontrivial asymptotic behavior, and if $\nu$ is the asymptotic eigenvalue, then 
\begin{equation*}
\nu < \delta \ \Rightarrow \ \wind_\infty(\util) \leq \alpha^{<\delta}_\sigma(P)
\end{equation*}
where $\sigma$ is aligned with $\util$.

Finally, we need to recall invariants $\wind_\pi$ and $\wind_\infty$ introduced in~\cite{props2} for a finite-energy curve in $(\R\times V^\pm,\jtil)$. Suppose that the domain is connected and the curve $\vtil=(b,v)$ is not a (possibly branched) cover of a trivial cylinder. In this case the curve has nontrivial asymptotic behavior at every puncture; see Remark~\ref{rmk_asymp_behaviour}. Choose a $d\lambda$-symplectic trivializing frame of $v^*\xi$ that extends to a collection $\sigma$ of trivializing frames over the asymptotic limits. The set of punctures splits as $\Gamma^+ \cup \Gamma^-$ into positive and negative punctures, and for each $z \in \Gamma^\pm$ we denote by $\nu_z$ the asymptotic eigenvalue of $\vtil$ at the puncture $z$. Denote $\wind_\infty(\vtil,z,\sigma) = \wind_\sigma(\nu_z)$. HWZ~\cite{props2} define
\begin{equation}
\wind_\infty(\vtil) = \sum_{z\in\Gamma^+} \wind_\infty(\vtil,z,\sigma) - \sum_{z\in\Gamma^-} \wind_\infty(\vtil,z,\sigma)
\end{equation}
\begin{equation}
\wind_\pi(\vtil) = \wind_\infty(\vtil) - \chi + \#\Gamma^+ + \#\Gamma^-
\end{equation}
where $\chi$ is the Euler characteristic of the underlying closed Riemann surface. It is proved in~\cite{props2} that
\begin{equation}
\wind_\pi(\vtil) \geq 0 \qquad \text{and} \qquad \wind_\pi(\vtil) = 0 \Rightarrow \text{$v$ is an immersion.}
\end{equation}

\subsection{Fredholm indices and moduli spaces}

For a given $P \in \P$ we denote by
\begin{equation}
\label{defn_Mfast}
\Mfast(P,\jbar)
\end{equation}
the set of equivalence classes of embedded fast finite-energy planes $\util:(\C,i) \to (\overline{W},\jbar)$ with a positive puncture at $\infty$ where it is asymptotic to $P$ satisfying $\CZ(\util)\geq 3$. Here two such planes $\util_0,\util_1$ are said to be equivalent if there exist $A\in\C^*$ and $B\in\C$ such that $\util_1(z)=\util_0(Az+B)$. An element of $\Mfast(P,\jbar)$ represented by a plane $\util$ will be denoted by $[\util]$. Similarly one can consider $\Mfast_1(P,\jbar)$ the set of equivalence classes of pairs $(\util,z)$ where $\util:(\C,i)\to(\overline{W},\jbar)$ is an embedded fast finite-energy plane with a positive puncture, satisfying $\CZ(\util)\geq3$, and $z\in\C$. Two pairs $(\util_0,z_0)$, $(\util_1,z_1)$ are said to be equivalent if there exists $(A,B)\in\C^*\times\C$ such that $\util_1(z)=\util_0(Az+B)$ holds for all $z \in \C$, and that $z_0=Az_1+B$.

We can set up a Fredholm theory for $\Mfast(P,\jbar)$ modeled on sections of the normal bundle, using Sobolev spaces or Hölder spaces. Consider an element of $\Mfast(P,\jbar)$ represented by a fast embedded plane $\util$. One can look at sections of the normal bundle of $\util(\C)$ with exponential decay faster than $\delta<0$, where $\delta$ is placed precisely in the spectral gap between eigenvalues of the asymptotic operator satisfying $\wind_\sigma=1$ and $\wind_\sigma=2$, where $\sigma$ is a $d\lambda$-symplectic trivialization of $x(T\cdot)^*\xi$ aligned with $\util$ in the sense explained in \S~\ref{ssec_fast_planes}. This is possible precisely if $\CZ(\util) = \CZ^0_\sigma(P)\geq 3$. Note that $\alpha^{<\delta}_\sigma(P) = 1$. The Fredholm index of the linearization $D_{\util}$ of $\bar\partial_{\jbar}$ at $\util$ restricted to this space of sections is
\begin{equation}
{\rm ind}_\delta(\util) = \CZ^\delta_\sigma(P) - 1 = 3-1 = 2.
\end{equation}

\begin{remark}
The weight $\delta<0$ will be different for two embedded fast planes asymptotic to~$P$ that form a sphere where $c_1(TW,\omega)$ does not vanish.
\end{remark}

An important consequence is that we have \textit{automatic transversality}, i.e. $D_{\util}$ at an embedded fast plane $\util$ is always a surjective Fredholm operator. This can be proved by an indirect argument as follows. There is no loss of generality to deform the normal bundle so that it coincides with $u^*\xi$ over $\C\setminus B_R(0)$, $R\gg1$. A global trivializing section of the normal bundle then induces, up to homotopy, a $d\lambda$-symplectic trivialization $\sigma_N$ of $x(T\cdot)^*\xi$ which winds $+1$ with respect to $\sigma$. This means that $\wind_{\sigma_N}(\mu) = \wind_{\sigma}(\mu)-1$ for every eigenvalue $\mu$ of the asymptotic operator. Moreover, a nontrivial section $\zeta\in\ker D_{\util}$ admits an asymptotic behavior governed by an eigensection of the asymptotic operator associated to an eigenvalue $\nu<\delta$, see Theorem~6.1 in~\cite{fast} or Theorem~A.1 in~\cite{sie_CPAM}. Hence, $\zeta$ does not vanish near $\infty$ and the total algebraic count of zeros of $\zeta$ is equal to $\wind_{\sigma_N}(\nu) = \wind_{\sigma}(\nu)-1 \leq \alpha^{<\delta}_\sigma(P)-1=1-1=0$. But the equation $D_{\util}\zeta=0$ allows us to use Carleman's similarity principle to say that zeros are isolated and count positively. The important conclusion that $\zeta$ never vanishes. Since the Fredholm index is $2$, we would find $3$ linearly independent sections of the kernel if the linearization is not surjective. But the normal bundle is two-dimensional, hence a nontrivial linear combination of them would have to vanish at some point, contradiction.

The consequence of the above arguments and remarks is that $\Mfast(P,\jbar)$ is a smooth manifold of dimension two when equipped with the topology inherited from the functional analytic set-up used for the Fredholm theory, for any $\jbar$, provided $\lambda$ is nondegenerate up to action $T$. Similarly, $\Mfast_1(P,\jbar)$ is a smooth manifold of dimension four, for any $\jbar$, provided $\lambda$ is nondegenerate up to action~$T$.

\begin{remark}
It turns out that, under our standing assumption that $\lambda$ is nondegenerate, the topologies on $\Mfast(P,\jbar)$ and on $\Mfast_1(P,\jbar)$ inherited from the functional analytic set-up used for the Fredholm theory coincides with the topology of $C^\infty_{\rm loc}$-convergence. There are situations where this can be proved dropping the assumption that $\lambda$ is nondegenerate~\cite{hryn_jsg},\cite{elliptic}.
\end{remark}

By the above remark and automatic transversality, it follows that for every $\jbar$ the space $\Mfast_1(P,\jbar)$ is a smooth $4$-dimensional manifold equipped with a smooth evaluation map
\begin{equation}\label{ev_map}
\ev:\Mfast_1(P,\jbar) \to \overline{W},\qquad \ev[\util,z]=\util(z).
\end{equation}

\begin{remark}
\label{rmk_ev_submersion}
The map $\ev$ is a local diffeomorphism: this is a direct consequence of the facts that planes in $\Mfast_1(P,\jbar)$ are embedded, and that sections in the kernel of the linearized Cauchy-Riemann operator with the previously described exponential decay never vanish; see the argument for this last claim above.
\end{remark}

\subsection{Intersection theory}

Our first goal here is to prove the following statement.

\begin{lemma}\label{lemma_intersections_fast_planes}
If $[\util],[\vtil]$ belong to the same connected component of $\Mfast(P,\jbar)$ then $\util(\C) \cap \vtil(\C) \neq \emptyset$ if, and only if, $[\util]=[\vtil]$.
\end{lemma}

The proof uses a weighted version of the intersection number from~\cite{sief_int}. 
Let us fix $e\in \pi_2(\W_\infty,\{+\infty\}\times P)$ and denote by $$ \Mfast(P,\jbar,e) \subset \Mfast(P,\jbar) $$ the subset of planes representing $e$. 
The weight $\delta<0$ is placed on the special spectral gap of the asymptotic operator of $(P,J)$ as explained before, and can be taken the same for all planes in $\Mfast(P,\jbar,e)$. 
Let $\tau$ be any trivializing $d\lambda$-symplectic frame of $x(T\cdot)^*\xi$. If $\vtil$ belongs to $\Mfast(P,\jbar,e)$ then we can push $\vtil$ near its positive puncture in the direction of $\tau$ to obtain a map $\vtil^\tau$ that does not intersect $\vtil$ near the puncture. More precisely, choose the Martinet tube so that in coordinates $(\vartheta,z=x_1+ix_2)$ the frame $\{\frac{\partial}{\partial {x_1}},\frac{\partial}{\partial {x_2}}\}$ along $P$ does not wind with respect to $\tau$. Use positive holomorphic polar coordinates $(s,t)$ at the puncture to represent as $\vtil(s,t)$ as a map $(a(s,t),\vartheta(s,t),z(s,t))$. Then, with a nondecreasing bump function $\beta:\R\to[0,1]$ supported on $[s_0,+\infty)$, $s_0\gg1$, equal to $1$ near $+\infty$, define $\vtil^\tau$ by $\vtil^\tau(s,t) = (a(s,t),\vartheta(s,t),z(s,t)+\beta(s))$ near the puncture, and $\vtil^\tau=\vtil$ away from the puncture. Finally we define
\begin{equation}
\util*_\delta\vtil = {\rm int}(\util,\vtil^\tau) + \alpha^{<\delta}_\tau(P)
\end{equation}
for $[\util],[\vtil] \in \Mfast(P,\jbar,e)$. Here ${\rm int}$ denotes the intersection number, where $\W$ is oriented by $\overline{\omega}\wedge \overline{\omega}$ and the domains of $\util$ and $\vtil^\tau$ carry the complex orientation. One checks that $\util*_\delta\vtil$ does not depend on the choice of $\tau$. Note that $\util*_\delta\vtil$ depends only on the corresponding classes in $\Mfast(P,\jbar,e)$. Arguing as in~\cite{sief_int} one establishes the following lemma which summarizes some of the main properties of this intersection number.

\begin{lemma}\label{lemma_int_invariance}
The number $\util*_\delta\vtil$ is nonnegative and does not change when $[\util],[\vtil]$ vary continuously on $\Mfast(P,\jbar,e)$. Moreover, if $[\util],[\vtil] \in \Mfast(P,\jbar,e)$ satisfy $\util(\C) \neq \vtil(\C)$ and $\util(\C) \cap \vtil(\C) \neq \emptyset$ then $\util*_\delta\vtil>0$.
\end{lemma}

Using the frame $\sigma_N$ induced by a global trivializing frame of the normal bundle, one computes for any $\util \in \Mfast(P,\jbar,e)$
\begin{equation}\label{int_u_u}
\util*_\delta\util = {\rm int}(\util,\util^{\sigma_N}) + \alpha^{<\delta}_{\sigma_N}(P) = 0+0 = 0
\end{equation}
Lemma~\ref{lemma_intersections_fast_planes} follows as a direct consequence of $\util*_\delta\util=0$ and Lemma~\ref{lemma_int_invariance}. 

\begin{lemma}
\label{lemma_limits_of_fast_planes}
Let $\util_n:(\C,i) \to (\W,\jbar)$ be fast finite-energy planes asymptotic to a simply covered periodic Reeb orbit $P = (x,T)$, that define elements of $\Mfast(P,\bar J,e)$.
Let $\sigma$ be a (unique up to homotopy) $d\lambda$-symplectic trivialization of $x(T\cdot)^*\xi$ aligned with $\util_n$ $\forall n$.
Let $\Gamma \subset \C$ be finite, and assume that $\util_n$ converges in $C^\infty_{\rm loc}(\C\setminus\Gamma)$ to a finite-energy map $\util:(\C \setminus\Gamma,i) \to (\W,\jbar)$ asymptotic to $P$ at~$\infty$.
If $\util$ has a nontrivial asymptotic behavior at $\infty$ with asymptotic eigenvalue $\nu$ then $\wind_\sigma(\nu) \leq 1$. In particular, if $\Gamma=\emptyset$ then $\util$ is fast plane. 
\end{lemma}

\begin{proof}
Write $\util_n(s,t)$ instead of $\util_n(e^{2\pi(s+it)})$, and similarly $\util(s,t)$ for $s$ large enough. Let $\mathcal{N}$ be a small tubular neighborhood of $x(\R)$ in $V^+$. We can assume, with no loss of generality, that $\sigma$ extends to a $d\lambda$-symplectic trivializing frame of $\xi|_{\mathcal{N}}$ which we still denote by $\sigma$. It follows from a suitable application of the Monotonicity Lemma,  Lemma~5.2 in~\cite{sft_comp} together with results on cylinders with small contact area from~\cite{small_area} that there exists $s_0\gg1$ such that 
$$ 
\util_n([s_0,+\infty)\times\R/\Z) \subset [0,+\infty)\times \mathcal{N}\ \forall n, \qquad \util([s_0,+\infty)\times\R/\Z) \subset [0,+\infty)\times \mathcal{N} \, .
$$
Write $\util_n=(a_n,u_n)$ and $\util=(a,u)$ in components, for $s\geq s_0$. Let $\pi_\xi:TV^+ \to \xi$ denote the bundle projection along $\R X_\lambda$. From Theorem~\ref{thm_asymptotics} it follows that if $s_1>s_0$ is fixed large enough then $\pi_\xi(\partial_su)$ does not vanish on $[s_1,+\infty)\times\R/\Z$ and the winding number $\wind_\sigma(\pi_\xi(\partial_su)(s_1,\cdot))$ of $t\mapsto \pi_\xi(\partial_su)(s_1,t)$ in the frame $\sigma$ is equal to $\wind_\sigma(\nu)$. Since $\pi_\xi(\partial_su_n) \to \pi_\xi(\partial_su)$ in $C^\infty_{\rm loc}$ we find $n_0$ such that if $n\geq n_0$ then $\pi_\xi(\partial_su_n)$ does not vanish on $\{s_1\}\times\R/\Z$ and 
$$ 
\wind_\sigma(\pi_\xi(\partial_su_n)(s_1,\cdot)) = \wind_\sigma(\pi_\xi(\partial_su)(s_1,\cdot)) = \wind_\sigma(\nu) \, .
$$ 
Using $\sigma$ one can represent $(s,t) \mapsto \pi_\xi(\partial_su_n)$ by smooth maps $\zeta_n : [s_0,+\infty)\times \R / \Z \to \C$ satisfying a Cauchy-Riemann type equation. Carleman's similarity principle implies that either $\zeta_n$ vanishes identically on $(s_0,+\infty)\times\R/\Z$, or its zeros are isolated and count positively. Since $s_1>s_0$ and $\zeta_n(s_1,\cdot)$ does not vanish when $n\geq n_0$ we conclude that for every $n\geq n_0$ the zeroes of $\zeta_n$ on $(s_0,+\infty)\times \R/\Z$ form a discrete set. If $n\geq n_0$ and $\util_n$ has a trivial asymptotic behavior then $\pi_\xi(\partial_s\util_n)$ vanishes identically near the puncture, absurd. Hence $\util_n$ has nontrivial asymptotic behavior when $n\geq n_0$. Moreover, by Theorem~\ref{thm_asymptotics} $\zeta_n(s,t)$ does not vanish when $s$ is large enough and for every $n$ we have $$ \lim_{s\to+\infty} \wind(\zeta_n(s,\cdot)) = \lim_{s\to+\infty} \wind_\sigma(\pi_\xi(\partial_su_n)(s,\cdot)) = \wind_\sigma(\nu_n) \, . $$
Standard degree theory implies that $\wind_\sigma(\nu_n) - \wind(\zeta_n(s_1,\cdot))\geq 0$ is the algebraic count of zeros of $\zeta_n$ on $[s_1,+\infty)\times\R/\Z$, i.e. $\wind_\sigma(\pi_\xi(\partial_su_n)(s_1,\cdot)) \leq \wind_\sigma(\nu_n)$ for all $n\geq n_0$. Hence 
$$ 
n\geq n_0 \Rightarrow \wind_\sigma(\nu) \leq \wind_\sigma(\nu_n) \leq 1 
$$ 
as desired.
\end{proof}

\begin{lemma}
\label{lemma_limits_in_moduli_space}
If $[\util_n] \in \Mfast(P,\jbar,e)$ is a sequence such that $\util_n$ $C^\infty_{\rm loc}$-converges to a plane $\util$ asymptotic to $P$, then $\util$ defines an element of $\Mfast(P,\jbar,e)$.
\end{lemma}

\begin{proof}
That $\util$ represents $e$ follows from the convergence. 
Note that $\util$ is somewhere injective since its asymptotic limit $P$ is simply covered. 
If $\util$ has a critical point or a self-intersection then we can invoke~\cite{MW} to conclude that $\util_n$ is not embedded when $n$ is large enough, absurd. 
The somewhere injectivity of $\util$ is crucial here. 
Hence $\util$ is embedded.
By Lemma~\ref{lemma_limits_of_fast_planes} $\util$ is fast. 
Thus $\util \in \Mfast(P,\jbar,e)$. 
\end{proof}

\subsection{SFT Compactness}

The SFT compactness theorem from~\cite{sft_comp} is the generalization of Gromov's compactness theorem to curves with punctures. The first step in describing its statement is a discussion of \textit{nodal curves} in $(\overline{W},\overline{\omega},\jbar)$ or in $(\R\times V,d(e^a\lambda),\jtil)$. A nodal curve in $(\overline{W},\overline{\omega},\jbar)$ is an equivalence class of tuples $(\util,S,j,\Gamma_+,\Gamma_-,D)$, where 
\begin{itemize}
\item $(S,j)$ is a closed Riemann surface,
\item $\Gamma_+$ and $\Gamma_-$ are disjoint ordered finite subsets of $S$,
\item $\util:(S\setminus(\Gamma_+\cup\Gamma_-),j) \to (\overline{W},\jbar)$ is a finite-energy pseudo-holomorphic map with positive punctures on $\Gamma_+$ and negative punctures on $\Gamma_-$,
\item $D$ is a finite unordered set of pairwise disjoint unordered pairs of distinct points of $S\setminus(\Gamma_+\cup\Gamma_-)$, such that if $\{z,w\} \in D$ then $\util(z)=\util(w)$. 
\end{itemize}
Note that $S$ is not assumed to be connected. The pairs in $D$ are called \textit{nodal pairs} and, at times, we might also denote by $D$ the subset of $S\setminus(\Gamma_+\cup\Gamma_-)$ consisting of all points forming the nodal pairs. Two such tuples are declared equivalent $$ (\util,S,j,\Gamma_+,\Gamma_-,D) \sim (\util',S',j',\Gamma_+',\Gamma_-',D') $$ if there exists a biholomorphism $\phi : (S,j) \to (S',j')$ such that 
\begin{itemize}
\item $\phi(\Gamma_\pm)=\Gamma_\pm'$ and $\phi$ defines order preserving maps $\Gamma_\pm \to \Gamma'_\pm$,
\item $\phi(D) = D'$ and $\phi$ maps pairs to pairs,
\item $\util = \util'\circ\phi$.
\end{itemize}
We may refer to $\phi$ simply as a \textit{reparametrization}. Nodal curves in $(\R\times V,d(e^a\lambda),\jtil)$ are defined in the same manner except that one needs to further quotient by the action of $(\R,+)$ on the first component. In both cases, equivalence classes will be denoted by $[\util,S,j,\Gamma_+,\Gamma_-,D]$. The nodal curve is said to be \textit{smooth} if $D=\emptyset$.

A nodal curve $[\util,S,j,\Gamma_+,\Gamma_-,D]$ in $(\overline{W},\jbar)$ is \textit{stable} if $2g_*+\mu_*\geq 3$ holds for every connected component $S_* \subset S$ such that $\util|_{S_*\setminus(\Gamma_+\cup\Gamma_-)}$ is the constant map; here $g_*$ is the genus of $S_*$ and $\mu_*$ is total number of punctures and nodal points in $S_*$. For a nodal curve in $(\R\times V,\jtil)$ stability is defined by further asking the existence of at least one connected component $S_0 \subset S$ such that $\util|_{S_0\setminus(\Gamma_+\cup\Gamma_-)}$ is not an unbranched cover of a trivial cylinder over a periodic orbit.

\begin{remark}
\label{rmk_blown_up_circles}
If $(S,j)$ is a Riemann surface and $z\in S$ then the circle $(T_zS\setminus0)/\R_+$ is naturally induced with a metric which makes it isometric to the standard $\R/2\pi\Z$, and with an orientation. We will refer to this circle as the \textit{blown up circle at $z$}. The punctured surface $S\setminus\{z\}$ may be compactified to a surface with boundary obtained by adding a blown up circle at $z$.
\end{remark}

The next step is to consider the notion of a \textit{holomorphic building} $\mathbf{u}$ of height $k_-|1|k_+$, with $k_\pm\geq0$. The building $\mathbf{u}$ is the equivalence class of tuples $$ \{\{\util_m\}_{-k_-\leq m\leq k_+},\{\Phi_m\}_{-k_-\leq m\leq k_+-1}\} $$ where:
\begin{itemize}
\item[(i)] $\util_0 = (\util_0,S_0,j_0,\Gamma^0_+,\Gamma^0_-,D_0)$ represents a nodal curve in $(\overline{W},\jbar)$.
\item[(ii)] $\forall m\geq 1$, $\util_m = (\util_m,S_m,j_m,\Gamma^m_+,\Gamma^m_-,D_m)$ represents a nodal curve in $(\R\times V^+,\jtil)$.
\item[(iii)] $\forall m\leq -1$, $\util_m = (\util_m,S_m,j_m,\Gamma^m_+,\Gamma^m_-,D_m)$ represents a nodal curve in $(\R\times V^-,\jtil)$.
\item[(iv)] The $\Phi_m$ are orientation reversing isometries $$ \Phi_m : \bigcup_{z\in\Gamma^m_+} (T_zS_m\setminus0)/\R_+ \to \bigcup_{z\in\Gamma^{m+1}_-} (T_zS_{m+1}\setminus0)/\R_+ $$ that cover order preserving bijections $\Gamma^m_+ \to \Gamma^{m+1}_-$, and such that the following holds. When $m\neq 0$ write $\util_m=(a_m,u_m)$. Add to $S_m \setminus (\Gamma^m_+\cup \Gamma^m_-)$ the blown-up circles $\{C_z=(T_zS^m\setminus0)/\R_+\}_{z\in \Gamma^m_+}$, at the positive punctures to obtain a surface with boundary $\widehat{S^m}$. The asymptotic behavior allows for the  projected map $u_m$ to be uniquely extended to a map $\overline{u}_m$ on $\widehat{S^m}$. Analogously, when $m<k_+$, add to  $S_{m+1} \setminus (\Gamma^{m+1}_+\cup \Gamma^{m+1}_-)$ the blown-up circles at the negative punctures $\Gamma^{m+1}_-$ to obtain a surface $\widehat{S^{m+1}}$ and an extended map $\overline{u}_{m+1}$. Then it is required that $\overline{u}_{m+1} \circ \Phi_m = \overline{u}_m$ on $C_z$ for every $z\in\Gamma_m^+$ and every $-k_-\leq m\leq k_+-1$.
\end{itemize}

Two such collections $\{\{\util_m\},\{\Phi_m\}\}$, $\{\{\util_m'\},\{\Phi_m'\}\}$ are declared equivalent if they have the same height, $\util_m'$ represents the same nodal curve as $\util_m$ for each $m$, and the corresponding reparameterizations intertwine the orientation reversing isometries $\Phi_m$ and $\Phi_m'$. Moreover, synchronized reordering of the intermediate punctures also define equivalent buildings. The data $\{\Phi_m\}$ induces a \textit{decoration} at the punctures between levels. 

Let $\mathbf{u}$ be a holomorphic building of height $k_-|1|k_+$ represented by $\{\{\util_m\},\{\Phi_m\}\}$ as described above. Fix an arbitrary choice $r$ of orientation reversing isometries between blown up circles of points in nodal pairs in $\cup_mD_m$. The data $r$ is called a \textit{decoration} at nodal pairs. Consider the surface $$ S^{\mathbf{u},r}  = \left( \sqcup_m \overline{S}_m \right) / \sim $$ where blown up circles at points of nodal pairs are identified by $r$, and $\Phi_m$ is used to identify blown up circles at $\Gamma^m_+$ with blown up circles at $\Gamma^{m+1}_-$. The interior circles of $S^{\mathbf{u},r}$ corresponding to blown up circles at nodal pairs or at punctures between levels will be called \textit{special circles}. Note that $S_m \setminus (\Gamma^m_+ \cup \Gamma^m_- \cup D_m)$ can be seen as an open subset of $S^{\mathbf{u},r}$. By the asymptotic behavior and conditions (i)-(iv) we can define a continuous map
\begin{equation}
F_{\mathbf{u}}:S^{\mathbf{u},r} \to \overline{W}_\infty
\end{equation}
that agrees with $\util_0$ on $S_0\setminus(\Gamma^0_+ \cup \Gamma^0_- \cup D_0)$,  with $(+\infty,u_m)$ on $S_m \setminus (\Gamma^m_+ \cup \Gamma^m_- \cup D_m)$ when $m\geq1$, and with $(-\infty,u_m)$ on $S_m \setminus (\Gamma^m_+ \cup \Gamma^m_- \cup D_m)$ when $m\leq-1$. Here we wrote $\util_m=(a_m,u_m)$ for $m\neq0$.

The final step in this discussion is the description of SFT convergence of a sequence $C_n=[\vtil_n,\Sigma_n,i_n,Z^n_+,Z^n_-,\emptyset]$ of connected smooth curves in $(\overline{W},\overline{J})$ with energy and genus bounds. Such a sequence is said to \textit{SFT converge} to a building $\mathbf{u}$ as above if there exists a sequence of diffeomorphisms $$ \varphi_n : S^{\mathbf{u},r} \to \overline{\Sigma}_n $$ and finite ordered sets $K_n \subset \Sigma_n$, $K\subset \cup_mS_m$ such that the following holds:
\begin{itemize}
\item[(a)] $K_n$ is disjoint from the punctures in $\Sigma_n$.
\item[(b)] $K$ is disjoint from the all punctures and nodal points in $\cup_mS_m$.
\item[(c)] If $g_n$ is the genus of $\Sigma_n$ and $\nu_n = \#(Z^n_+\cup Z^n_-\cup K_n)$ then $2g_n+\nu_n\geq3$. If $S_*$ is a connected component of $\cup_mS_m$ with genus $g_*$, and $\nu_*$ is the total number of punctures, nodal points and points of $K$ in $S_*$, then $2g_*+\nu_* \geq 3$.
\item[(d)] $\varphi_n$ maps blown up circles at $\Gamma^{k_+}_+$ onto blown up circles at $Z^n_+$ covering an order preserving bijection $\Gamma^{k_+}_+ \to Z^n_+$, blown up circles at $\Gamma^{-k_-}_-$ onto blown up circles at $Z^n_-$ covering an order preserving bijection $\Gamma^{-k_-}_- \to Z^n_-$, and maps $K \to K_n$ in an order-preserving manner.
\item[(e)] Let $h_n$ be the hyperbolic metric on $\Sigma_n \setminus (Z^n_+\cup Z^n_- \cup K_n)$ induced by $i_n$, and $h$ be the hyperbolic metric on $\cup_m \ S_m \setminus (\Gamma^m_+\cup\Gamma^m_-\cup D_m \cup K)$ induced by $\{j_m\}$. Then $\varphi_n^*h_n \to h$ in the $C^\infty_{\rm loc}$-topology on $\cup_m \ S_m \setminus (\Gamma^m_+\cup\Gamma^m_-\cup D_m \cup K)$, where the latter is seen as an open subset of $S^{\mathbf{u},r}$. Moreover, $\varphi_n$ maps special circles to closed geodesics of $h_n$.
\item[(f)] $F_{C_n} \to F_{\mathbf{u}}$ in $C^0$, where here $C_n$ is seen as a building of height $0|1|0$.
\item[(g)] The following holds:
\begin{itemize}
\item[($+$)] If $m\geq1$ then there exists $c_{m,n} \to +\infty$ such that the following holds. For every compact set $X\subset S_m \setminus (\Gamma^m_+\cup\Gamma^m_-\cup D_m)$ and every $\varepsilon>0$ there exists $n_{X,\epsilon}$ such that if $n\geq n_{X,\varepsilon}$ then $\vtil_n \circ \varphi_n(X) \subset [0,+\infty) \times V^+$ and $$ \sup_{z\in X} |\pi_\R \circ \vtil_n \circ \varphi_n(z) - c_{m,n} - a_m(z)| \leq \varepsilon $$ where $\pi_\R$ denotes projection onto the $\R$-component. 
\item[($-$)] If $m\leq-1$ then there exists $c_{m,n} \to -\infty$ such that the following holds. For every compact set $X\subset S_m \setminus (\Gamma^m_+\cup\Gamma^m_-\cup D_m)$ and every $\varepsilon>0$ there exists $n_{X,\epsilon}$ such that if $n\geq n_{X,\varepsilon}$ then $\vtil_n \circ \varphi_n(X) \subset (-\infty,0] \times V^-$ and $$ \sup_{z\in X} |\pi_\R \circ \vtil_n \circ \varphi_n(z) - c_{m,n} - a_m(z)| \leq \varepsilon. $$
\end{itemize}
\end{itemize}

We also need to consider buildings and SFT convergence for curves in $(\R\times V^+,d(e^a\lambda),\jtil)$. The notion of nodal curves in this setting was already explained above. As for the buildings: one looks at collections $$ \{\{\util_m\}_{1\leq m\leq k_+},\{\Phi_m\}_{1\leq m\leq k_+-1}\} $$ as above where the $\util_m$ represent nodal curves in $(\R\times V^+,d(e^a\lambda),\jtil)$. 
In this case it is simpler to consider $F_{\mathbf{u}}$ defined from the $V^+$-components of the $\{\util_m\}$, so that $F_{\mathbf{u}}$ takes values on $V^+$. Conditions (a)-(f) remain unchanged, (g) is replaced by:
\begin{itemize}
\item[($\text{g}'$)] For each $m$ there exists $c_{m,n} \in\R$ such that the following holds. For every compact set $X\subset S_m \setminus (\Gamma^m_+\cup\Gamma^m_-\cup D_m)$ and every $\varepsilon>0$ there exists $n_{X,\epsilon}$ such that if $n\geq n_{X,\varepsilon}$ then
$$ \sup_{z\in X} |\pi_\R \circ \vtil_n \circ \varphi_n(z) - c_{m,n} - a_m(z)| \leq \varepsilon $$ where $\pi_\R$ denotes projection onto the $\R$-component. 
\end{itemize}

Since $\lambda$ is nondegenerate, it follows from the SFT compactness theorem~\cite{sft_comp} that any sequence of smooth curves intersecting a given compact subset of $\overline{W}$, with energy and genus bounds, contains a subsequence that SFT converges to a building as described above. 

\subsection{SFT-limits of fast planes}

Let $P=(x,T)$ be a simply covered periodic Reeb orbit  in~$V^+$. Consider a sequence $\vtil_n : (\C,i) \to (\overline{W},\jbar)$ of embedded fast planes in $(\overline{W},\overline{J})$ asymptotic to~$P$, satisfying~$\CZ(\vtil_n)\geq 3$, defining curves in the same connected component of $\Mfast(P,\jbar)$.  The goal here is to describe properties of the limiting holomorphic building of an SFT-convergent subsequence of these planes, assuming that they go through a fixed compact set $E\subset\overline{W}$.

\begin{remark}\label{rmk_hom_classes}
If $P'=(x',T')$ is a closed Reeb orbit in $V^+$ then from now on we identify $\pi_2(W,x'(\R)) \simeq \pi_2(\W_\infty,\{+\infty\}\times x'(\R))$ via an isomorphism $e\mapsto \hat e$ defined as follows. If $e$ is represented by $U:\D\to W$ such that $U(e^{i2\pi t}) \subset x'(\R)$ then $\hat e$ is represented by a map $\hat U$ from $(\D\sqcup [0,+\infty]\times\R/\Z)/[e^{i2\pi t}\sim (0,t)]$ to $\W_\infty$ defined by $\hat U(z)=U(z)$ if $z\in\D$, and $\hat U(s,t) = (s,U(e^{i2\pi t}))$ if $(s,t) \in [0,+\infty]\times\R/\Z$. It is simple to show that $e\mapsto \hat e$ is an isomorphism, and to write a formula for its inverse.
\end{remark}

Let $e \in \pi_2(W,x(\R))$ map to the $\lambda$-positive generator of $\pi_1(x(\R))$ by the boundary map. Let $w_n\in \C$ and assume that $\ev[\vtil_n,w_n] \in E$ for all $n$. Note that $[\vtil_n,w_n] \in \Mfast_1(P,\jbar)$. Up to selecting a subsequence, it can be assumed that these planes SFT-converge to a holomorphic building $\mathbf{u}$ of height $k_-|1|k_+$
\begin{equation}
\label{limiting_building}
C_n = [\vtil_n,\C\cup\{\infty\},i,\{\infty\},\emptyset,\emptyset] \overset{n\to\infty}{\longrightarrow} {\bf u}
\end{equation}
and that $F_{C_n}$ represents the class $e$ for every $n$. 

For any fixed $b>0$ the symplectic area
\begin{equation}
\int_\C \vtil_n^*\overline{\omega}_b
\end{equation}
is independent of $n$, where $\overline{\omega}_b$ is the symplectic form defined as in~\eqref{symp_area_b}. 

\begin{lemma}\label{lemma_comp_forcing_levels_above}
Suppose $\ev[\vtil_n,z_n] \to (+\infty,p) \in \overline{W}_\infty$ for some sequence $z_n \in \C$. If $p\not\in x(\R)$ then $k_+\geq1$. 
\end{lemma}

\begin{lemma}\label{lemma_comp_conseq_levels_above}
If $k_+\geq1$ then there exists a periodic Reeb orbit $P'=(x',T')$ in $V^+$ such that $T'<T$, $P'$ spans a capping disk $D'$ in $W$ that does not intersect some disk in class $e$, and the Conley-Zehnder index of $P'$ relative to $D'$ is $\leq2$.
\end{lemma}

\begin{lemma}\label{lemma_comp_spheres_zero_level}
Let $[\util_0,S_0,j_0,\Gamma^+_0,\Gamma^-_0,D_0]$ be the nodal curve at the level zero of $\mathbf{u}$. If $D_0 \neq \emptyset$, $k_+=k_-=0$ then there exists non-constant $\jbar$-sphere.
\end{lemma}

\begin{lemma}\label{lemma_comp_periods_lower_levels}
If $k_-\geq1$ then there exists a periodic Reeb orbit $P_*=(x_*,T_*)$ in $(V^-,\lambda)$ satisfying
\begin{equation}
\label{action_ineq_lemma}
T_* \leq e^b \ \left( \int_\C \vtil_n^*\overline{\omega}_b \ - \ \int_{S_0\setminus(\Gamma^+_0\cup\Gamma^-_0)} \util_0^*\overline{\omega}_b \right)
\end{equation}
for every $b>0$.
\end{lemma}

\begin{proof}
[Proof of Lemma~\ref{lemma_comp_forcing_levels_above}]
Assume, by contradiction, that $k_+=0$. By SFT convergence, specifically by condition (f), we have $$ (+\infty,p) \in F_{\mathbf{u}}(S^{\mathbf{u},r}) $$  from where it follows that $F_{\mathbf{u}}(S^{\mathbf{u},r}) \cap (\{+\infty\} \times V^+) = \{+\infty\}\times x(\R)$ and $p\in x(\R)$, absurd.
\end{proof}

\begin{proof}
[Proof of Lemma~\ref{lemma_comp_conseq_levels_above}]
Let $c_1,\dots,c_h$ be the special circles of $S^{\mathbf{u},r}$ corresponding to the negative punctures $\Gamma^-_{1}$ of $\util_{1}$. Note that $S^{\mathbf{u},r}$ is a disk and each $c_j$ is contained in its interior. Note also that the $c_j$ necessarily bound pairwise disjoint closed disks $D_j$ in the interior of $S^{\mathbf{u},r}$ since otherwise there would be a curve in some level $m\geq1$ without positive punctures, which is impossible.

Consider the building $\mathbf{u}_+$ on $\R\times V^+$ formed by the positive levels of $\mathbf{u}$, and denote by $r_+$ the restriction of $r$ to the nodal pairs in the positive levels. Then, by construction, we have $$ S^{\mathbf{u}_+,r_+} = S^{\mathbf{u},r} \setminus (\mathring{D}_1 \cup \dots \cup \mathring{D}_h) \qquad\qquad F_{\mathbf{u}}|_{S^{\mathbf{u}_+,r_+}} = \{+\infty\} \times F_{\mathbf{u}_+} $$ since $F_{\mathbf{u}_+}$ takes values on $V^+$. The top level $\util_{k_+}$ consists of one finite-energy sphere with precisely one positive puncture $\{z_+\} = \Gamma^+_{k_+}$ where it is asymptotic to $P$. Hence $\util_{k_+}$ does not cover a trivial cylinder and has nontrivial asymptotic behavior at its punctures. Let $\sigma$ be a $d\lambda$-symplectic trivialization of $x(T\cdot)^*\xi$ aligned with the~$\vtil_n$; this is independent of $n$ since all $\vtil_n$ represent the same class $e$ in $\pi_2(W,x(\R))$.

The inequality $\wind_\infty(\util_{k_+},z_+,\sigma) \leq \wind_\infty(\vtil_n) \leq 1$ follows as in the proof of Lemma~\ref{lemma_limits_of_fast_planes}. Moreover, $\sigma$ extends to a $d\lambda$-symplectic trivialization of $(F_{\mathbf{u}_+})^*\xi$ still denoted by $\sigma$. We shall now prove that for every $m\geq1$ the curve $\util_m$ has a negative puncture $z^*_m$ where it is asymptotic to a closed Reeb orbit $P^*$ satisfying $\CZ^0_\sigma(P^*)\leq2$. We start with the top level, and let $\zeta_1,\dots,\zeta_N$ be the negative punctures of $\util_{k_+}$, and $P_j$ be the asymptotic limit of $\util_{k_+}$ at $\zeta_j$. If $\CZ^0_\sigma(P_j) \geq 3$ then $\wind_\infty(\util_{k_+},\zeta_j,\sigma) \geq 2$. Arguing by contradiction, assume that $\CZ^0_\sigma(P_j) \geq 3$ for all~$j$. Then
\begin{equation*}
\begin{aligned} 
0 \leq \wind_\pi(\util_{k_+}) &= \wind_\infty(\util_{k_+}) - 2 + N + 1 \\
& \leq N - \sum_{j=1}^N \wind_\infty(\util_{k_+},\zeta_j,\sigma) \leq N-2N = -N
\end{aligned}
\end{equation*}
from where it follows that $N=0$, absurd. Hence we find the desired $z^*_{k_+} \in \Gamma^-_{k_+}$. Now one can proceed inductively, estimating as above. In fact, all connected components $\wtil$ of all levels $\util_{m}$, $m\geq1$, have at most one positive puncture $\zeta_+$. Assume that the asymptotic limit $P_+$ of $\wtil$ at $\zeta_+$ satisfies $\CZ^0_\sigma(P_+)\leq 2$. If $\wtil$ has trivial asymptotic behavior at $\zeta_+$ then one can use Carleman's similarity principle to conclude that $\wtil$ is a (possibly branched) cover of a trivial cylinder. In this case $P_+$ and the various asymptotic limits $P_-$ at the negative punctures cover the same primitive Reeb orbit, the covering multiplicity of $P_+$ being at least equal to that of all $P_-$. It follows, in this case, that $\CZ_\sigma^0(P_-)\leq 2$ at all negative punctures. If $\wtil$ has nontrivial asymptotic behavior at $\zeta_+$ then $\CZ^0_\sigma(P_+)\leq 2$ implies that $\wind_\infty(\wtil,\zeta_+,\sigma) \leq 1$ and one can argue as we did for the top level.

By the above argument we can assume, without loss of generality, that the asymptotic limit $P'=(x',T')$ of $\util_1$ at the negative puncture corresponding to $c_1$ satisfies $\CZ^0_\sigma(P')\leq 2$. It follows that $P'$ is geometrically distinct from $P$. Consider the disk $D \subset S^{\mathbf{u},r}$ spanned by $c_1$. Note that $D$ is contained in the interior of the larger disk $S^{\mathbf{u},r}$ which is naturally oriented by the conformal structures of the levels of $\mathbf{u}$. The orientation induced on $D$ orients its boundary $c_1$. With this orientation, $F_{\mathbf{u}}$ maps $c_1$ to $\{+\infty\} \times x'(\R)$ along the Reeb flow. If we remove special circles from the interior of $D$ then we are left with an open subset of $S^{\mathbf{u},r}$ equal to the union of a certain collection $\mathcal{V}$ of connected components of $\sqcup_{m\leq 0} S_m\setminus(\Gamma^+_m\cup\Gamma^-_m\cup D_m)$; the reason why we do not see connected components of $\sqcup_{m\geq1} S_m\setminus(\Gamma^+_m\cup\Gamma^-_m\cup D_m)$ is because this would force some curve on a level $m\geq1$ to have no positive punctures, absurd. Connected components in $\mathcal{V}$ contained in levels $m\leq-1$ are mapped by $F_{\mathbf{u}}$ to $\{-\infty\}\times V^-$, hence their images under $F_{\mathbf{u}}$ do not touch the images of the $\vtil_n$. If the image of $\util_0$ intersects the image of some $\vtil_{n_0}$ then, by stability and positivity of intersections, for all $n$ large enough the image of $\vtil_n$ intersects the image of $\vtil_{n_0}$. Here use the fact that to no connected component of its domain $\util_0$ restricts a reparametrization of the $\vtil_{n_0}$: for action reasons, every asymptotic limit at a positive puncture of $\util_0$ has action strictly less than $T$ (the $d\lambda$-area of the top level is positive). This is a contradiction to Lemma~\ref{lemma_intersections_fast_planes}, and shows that $F_{\mathbf{u}}|_{D}$ defines a capping disk for $P'$ in $\W_\infty$ that does not intersect the image of $\vtil_n$, for any $n$. By construction, the Conley-Zehnder index of $P'$ relative to $F_{\mathbf{u}}|_{D}$ is $\CZ^0_\sigma(P') \leq 2$. The inequality $T'<T$ follows from Stokes theorem and the fact that the $d\lambda$-area of the top level is positive. 
\end{proof}

\begin{proof}
[Proof of Lemma~\ref{lemma_comp_spheres_zero_level}]
By assumption the entire limiting building ${\bf u}$ is equal to the nodal curve at level zero. One can describe this nodal curve as a plane $\Pi$ asymptotic to the simply covered Reeb orbit $P$, and several $\jbar$-spheres $S_1,\dots,S_k$ in $\overline{W}$ connected by nodes. Since the total genus is zero, these nodes lead to a non-constant sphere. By Lemma~\ref{lemma_limits_of_fast_planes} $\Pi$ is a fast plane.
\end{proof}

\begin{proof}
[Proof of Lemma~\ref{lemma_comp_periods_lower_levels}]
Let $c_1,\dots,c_h$ be the special circles of $S^{\mathbf{u},r}$ corresponding to the positive punctures $\Gamma^+_{-1}$ of $\util_{-1}$. Note that $S^{\mathbf{u},r}$ is a disk oriented by the conformal structures of the levels of $\mathbf{u}$. Each $c_j$ is contained in the interior of $S^{{\bf u},r}$. If $(W,\omega)$ is exact then these circles necessarily bound pairwise disjoint disks in $S^{\mathbf{u},r}$, but if not then we might see nested circles: this situation can only arise if there are curves without positive punctures in the level zero of $\mathbf{u}$ that are not connected via a nodal pair to any other curve with a positive puncture. 

In any case, choose $j_* \in \{1,\dots,h\}$ with the following property: $c_{j_*}$ corresponds to a negative puncture of a component of $\util_0$ that either has a positive puncture, or has no positive punctures but is connected to a component of $\util_0$ with a positive puncture by nodal pairs. Note that in the latter case there is exactly one such nodal pair due to the fact that $S^{{\bf u},r}$ has genus zero and that $(W,\omega)$ is symplectically aspherical. Such $j_*$ must exist, since otherwise $\util_0$ has no positive puncture, contradicting the fact that the entire building is the SFT-limit of a sequence of planes in $\overline{W}$ with one positive puncture. Let $P_*=(x_*,T_*)$ be the asymptotic limit of $\util_0$ at the negative puncture corresponding to $c_{j_*}$.

Consider the disk $D_*$ in the interior of $S^{{\bf u},r}$ bounded by $c_{j_*}$. Consider the set $J_*$ consisting of those $j\in\{1,\dots,h\}$ such that $j\neq j_*$, $c_j\subset D_*$ and the disk in $S^{{\bf u},r}$ spanned by $c_j$ intersects $S_0\setminus(\Gamma^+_0\cup\Gamma^-_0\cup D_0)$. The set $J_*$ might be empty. Let $D'_*$ be the set obtained by removing from $D_*$ the interiors of the disks spanned by the $\{c_j\}_{j\in J_*}$. The crucial property for us is that $D'_*$ contains no circle $c_j$ with $j\in J_*$ in its interior. When $D_*$ inherits the orientation from $S^{{\bf u},r}$ and  $c_{j_*}$ is oriented as the boundary of $D_*$, the map $F_{\mathbf{u}}$ maps $c_{j_*}$ to $\{-\infty\}\times x_*(\R)$ along the Reeb flow. If we remove special circles from the interior of $D'_*$ then we are left with an open subset of $S^{\mathbf{u},r}$ equal to the union of a certain collection $\mathcal{V}$ of connected components of $\sqcup_{m\leq-1} S_m\setminus(\Gamma^+_m\cup\Gamma^-_m\cup D_m)$; the reason why we do not see connected components of $S_0\setminus(\Gamma^+_0\cup\Gamma^-_0\cup D_0)$ is because this would force some circle $c_j$ with $j\in J_*$ to be contained in the interior of $D'_*$, but these do not exist. Denote by $Y \in \mathcal{V} \mapsto m(Y) \in \{-k_-,\dots,-1\}$ the function which assigns the level of $\mathbf{u}$ corresponding to $Y \in \mathcal{V}$.
 
For every $m\neq0$ we write in components $\util_m=(a_m,u_m)$. By the description of SFT convergence, and by Stokes theorem, we know that 
\begin{equation*}
\begin{aligned}
\int_\C \vtil_n^*\overline{\omega}_b &= e^b \left( \sum_{m\geq1} \int_{S_m\setminus(\Gamma^+_m\cup\Gamma^-_m)} u_m^*d\lambda \right) \\
&\qquad + \int_{S_0\setminus(\Gamma^+_0\cup\Gamma^-_0)} \util_0^*\overline{\omega}_b \ + \ e^{-b} \left( \sum_{m\leq-1} \int_{S_m\setminus(\Gamma^+_m\cup\Gamma^-_m)} u_m^*d\lambda \right)
\end{aligned}
\end{equation*}
and that 
\begin{equation*}
T_* \leq \sum_{Y \in \mathcal{V}} \int_Y u_{m(Y)}^*d\lambda
\end{equation*}
with strict inequality when $J_*\neq\emptyset$. One can now estimate
\begin{equation*}
e^{-b} \ T_* \leq e^{-b} \ \sum_{m\leq-1} \int_{S_m\setminus(\Gamma^+_m\cup\Gamma^-_m)} u_m^*d\lambda \leq \int_\C \vtil_n^*\overline{\omega}_b \ - \ \int_{S_0\setminus(\Gamma^+_0\cup\Gamma^-_0)} \util_0^*\overline{\omega}_b
\end{equation*}
as desired.
\end{proof}

\subsection{The symplectic capping disk as a finite-energy plane}

Let $P=(x,T)$ be a simply covered periodic Reeb orbit on $V^+$, and let $D$ be a symplectic slicing disk for $x(\R)$. Consider $e \in \pi_2(W,x(\R))$ the class represented by $D$. For the following statement we identify homotopy classes according to Remark~\ref{rmk_hom_classes}.

\begin{lemma}
\label{lemma_existence}
There exist $\jbar$ and $\vtil \in \Mfast(P,\jbar,e)$ such that $\wind_\infty(\vtil)=-\infty$.
\end{lemma}

The above lemma follows from the statement below.

\begin{lemma}
\label{lemma_transv_knot}
Consider a parametrized transverse knot $\gamma:\R/\Z\to M$ in $(M,\xi)$, i.e.\ $\gamma$ is an embedding satisfying $\lambda(\dot\gamma)>0$. Let $\varepsilon,\varepsilon'>0$ and let 
\[
\tilde u = (a,u):[-\varepsilon, 0]\times \R/\Z \to (-\varepsilon', 0]\times M
\]
be a smooth symplectic embedding satisfying $\util(0, t)=(0,\gamma(t))$. There exists a smooth symplectic embedding $\vtil:[-\varepsilon,+\infty) \times\R/\Z \to (-\varepsilon',+\infty)\times M$ such that
\begin{itemize}
\item $\vtil$ coincides with $\util$ on a neighborhood of $\{-\varepsilon\}\times \R/\Z$.
\item $\vtil(s,t) = (c(s+\varepsilon),\gamma(t))$ holds when $s$ is large enough, for some $c>0$.
\end{itemize}
\end{lemma}

\begin{proof}
We consider a map $\tilde u=(a,u):U\subset\R\times \R/\Z \to\R\times M$ where $U$ is some open subset of the cylinder. From
\begin{align*}
\tilde u^*(d(e^a\lambda))
&=\tilde u^*(e^a(da\wedge\lambda+d\lambda)) \\
&=e^a\left(a_s \lambda(u_t)-a_t\lambda(u_s)+d\lambda(u_s, u_t)\right)ds\wedge dt
\end{align*}
we conclude that such a map is symplectic precisely when
\begin{equation}\label{e:symp-condition}
a_s \lambda(u_t)-a_t\lambda(u_s)+d\lambda(u_s, u_t)>0.
\end{equation}

We now consider a parametrized transverse knot $\gamma:\R/\Z\to M$, i.e.\ $\gamma$ is an embedding satisfying $\lambda(\dot\gamma)>0$. Let
\[
\tilde u = (a,u):[-\varepsilon, 0]\times \R/\Z \to (-\varepsilon', 0]\times M
\]
be a smooth symplectic embedding satisfying
$u(0, t)=\gamma(t)$.
Shrinking $\varepsilon$ if necessary, we are free to assume that
$u(s,\cdot):{\R/\Z}\to M$ is a transverse knot for all $s\in [-\varepsilon, 0]$.

Choose a smooth map $\phi:[-\varepsilon, \infty)\to[-\varepsilon, 0]$ so that $\phi(-\varepsilon)=-\varepsilon$, and so that for some $s_0\in [-\varepsilon, \infty)$ the following holds: $\phi'(s)>0$ for all $-\varepsilon \leq s < s_0$, and $\phi(s)=0$ for all $s\ge s_0$. We note that $\phi$ maps $[-\varepsilon, s_0)$ diffeomorphically onto $[-\varepsilon, 0)$. Choose a cut-off function $\beta:(-\infty, 0]\to [0, 1]$ with $\beta\equiv 0$ on $(-\infty, -2\varepsilon/3]$, $\beta\equiv 1$ on $[-\varepsilon/3, 0]$, and $\beta'(s)>0$ on $(-2\varepsilon/3, -\varepsilon/3)$. For each $\sigma\ge 1$ we denote $\beta_{\sigma}(s)=\beta(\sigma s)$. For some $c>0$ to be fixed \textit{a posteriori}, define a function $b^\sigma:[-\varepsilon, \infty)\times {\R/\Z}\to (-\varepsilon', \infty)$ by
\begin{align*}
b^\sigma(s, t)
&=(1-\beta_{\sigma}(\phi(s)))a(\phi(s), t)+\beta_{\sigma}(\phi(s))c(s+\varepsilon) \\
&=a(\phi(s), t)+\beta_{\sigma}(\phi(s))(c(s+\varepsilon)-a(\phi(s), t)).
\end{align*}
We further define
$v:[-\varepsilon, \infty)\times {\R/\Z}\to M$ by 
\[
v(s, t)=u(\phi(s), t)
\]
and set
\[
\tilde v_\sigma=(b^\sigma, v):[-\varepsilon, \infty)\times {\R/\Z}\to [-\varepsilon', \infty)\times M.
\]
Direct calculations yield
\begin{align*}
b^\sigma_s(s, t) &= a_s(\phi(s), t)\phi'(s) + \beta_{\sigma}(\phi(s))(c-a_s(\phi(s), t)\phi'(s)) \\
& \qquad + \beta_{\sigma}'(\phi(s))\phi'(s)(c(s+\varepsilon)-a(\phi(s), t)) \\
b^\sigma_t(s, t) &= a_t(\phi(s), t)-\beta_{\sigma}(\phi(s))a_t(\phi(s), t) \\
v_s(s, t) &= u_s(\phi(s), t)\phi'(s) \\
v_t(s, t) &= u_t(\phi(s), t) 
\end{align*}
and thus we have
\begin{align*}
& (b^\sigma_s\lambda(v_t)-b^\sigma_t\lambda(v_s)+d\lambda(v_s, v_t))_{(s, t)} \\
& \qquad = \phi'(s)\left(a_s\lambda(u_t)-a_t\lambda(u_s)+d\lambda(u_s, u_t)\right)_{(\phi(s), t)} \\
& \qquad\qquad + \beta_{\sigma}'(\phi(s))\phi'(s)(c(s+\varepsilon)-a(\phi(s), t))\lambda(u_t)_{(\phi(s), t)} \\
& \qquad\qquad + \beta_{\sigma}(\phi(s))\big((c-a_s(\phi(s), t)\phi'(s)) \lambda(u_t)_{(\phi(s), t)}+\phi'(s)(a_t\lambda(u_s))_{(\phi(s), t)}\big).
\end{align*}
By the assumption that $\tilde u=(a, u)$ is a symplectic map and therefore satisfies \eqref{e:symp-condition},
the first term is always nonnegative and strictly positive 
where $\phi'(s)>0$ which is true exactly when
$s\in [-\varepsilon, s_0)$.
The second term is also always nonnegative since
$\beta_{\sigma}'$, $\phi'$, and $c(s+\varepsilon)$ are always nonnegative,
$a$ is everywhere nonpositive, and $\lambda(u_t)$ is everywhere positive.
Finally, the third term is always nonnegative and positive on the support of $\beta\circ\phi$ provided
\[
(c-a_s(\phi(s), t)\phi'(s)) \lambda(u_t)_{(\phi(s), t)}+\phi'(s)(a_t\lambda(u_s))_{(\phi(s), t)}>0
\]
or equivalently if
\[
c>\phi'(s)\left(
a_s(\phi(s), t)-(\lambda(u_t)_{(\phi(s),t)})^{-1}(a_t\lambda(u_s))_{(\phi(s), t)}
\right)
\]
which can always be arranged since the right-hand side has compact support. We note further that $\phi'(s)=0$ when $s>s_0$ in which case $\beta_{\sigma}(\phi(s))=\beta_{\sigma}(0)=1$ so the first and third terms are never simultaneously zero. We conclude that
\[
b^\sigma_s\lambda(v_t)-b^\sigma_t\lambda(v_s)+d\lambda(v_s, v_t)
\]
is everywhere positive, and thus that $\tilde v_\sigma$ is a symplectic immersion provided $c$ is chosen large enough.

It remains to show that $\tilde v_\sigma$ is an embedding for a suitable choice of $\sigma\ge 1$. Since $\tilde v_\sigma$
is clearly a proper map,
it suffices to show that
$\sigma$ can be chosen so that $\tilde v_{\sigma}$ is injective.
We consider then a pair of points
$(s_1, t_1)\ne (s_2, t_2)$ so that 
$\tilde v_{\sigma}(s_1, t_1)=\tilde v_{\sigma}(s_2, t_2)$,
or equivalently
\[
b^\sigma(s_1, t_1)=b^\sigma(s_2, t_2), \qquad u(\phi(s_1), t_1)=u(\phi(s_2), t_2).
\]
We note by the assumption that $u(\phi(s), \cdot)$ is an embedding the second condition can hold for distinct points only if $\phi(s_1)\ne \phi(s_2)$ and thus $s_1\ne s_2$. Without loss of generality we assume that $s_2>s_1$ and $\phi(s_2)> \phi(s_1)$. Hence, $\beta_\sigma(\phi(s_2))\ge \beta_\sigma(\phi(s_1))$. We note that for those values of $s$ for which $\beta_\sigma(\phi(s))=0$, i.e.\ those $s$ for which $\phi(s)\le -2\varepsilon/(3\sigma)$, we have that
\[
\tilde v_\sigma(s, t)=(a(\phi(s), t), u(\phi(s), t))=\tilde u(\phi(s), t)
\]
which is injective by the assumption that
$\tilde u$ is an embedding and that 
$\phi$ maps the interval $[-\varepsilon, s_0)$ diffeomorphically onto $[-\varepsilon, 0)$.
We conclude that $\beta_\sigma(\phi(s_2))>0$.
The condition 
$b^\sigma(s_1, t_1)=b^\sigma(s_2, t_2)$
then gives us that
\[
\begin{aligned}
& (1-\beta_{\sigma}(\phi(s_1)))a(\phi(s_1), t_1)+\beta_{\sigma}(\phi(s_1))c(s_1+\varepsilon) \\
& \qquad = (1-\beta_{\sigma}(\phi(s_2)))a(\phi(s_2), t_2)+\beta_{\sigma}(\phi(s_2))c(s_2+\varepsilon).
\end{aligned}
\]
which we can rearrange as
\begin{align*}
& \left[1-\beta_{\sigma}(\phi(s_1))\right]\left[a(\phi(s_1), t_1)-a(\phi(s_2), t_2)\right] \\
& \qquad = \left[\beta_{\sigma}(\phi(s_2))-\beta_{\sigma}(\phi(s_1))\right] \left[c(s_1+\varepsilon)-a(\phi(s_2),t_2)\right] +\beta_{\sigma}(\phi(s_2))c(s_2-s_1).
\end{align*}
We note that both terms on the right-hand side are nonnegative and that the second term is positive.
We conclude that the left-hand side must also be positive which, since $\beta_\sigma\le 1$ everywhere, it follows that
$\beta_\sigma(\phi(s_1))<1$ and that 
\[
a(\phi(s_1), t_1)>a(\phi(s_2), t_2).
\]

Now consider a sequence $\sigma_k\to\infty$
and assume that for every value of $k$ we can find points
$(s_k, t_k)\ne (S_k, T_k)$ for which 
$\tilde v_{\sigma_k}(s_k, t_k)=\tilde v_{\sigma_k}(S_k, T_k)$.
By the discussion from the previous paragraph we can assume without loss of generality that
$S_k>s_k$ and $\phi(S_k)>\phi(s_k)$, 
and consequently that $\beta_{\sigma_k}(\phi(S_k))>0$ and
$a(\phi(s_k), t_k)>a(\phi(S_k), T_k)$.
Passing to a subsequence if necessary we can assume that for some points $(\phi_\infty, t_\infty)$
and $(\Phi_\infty, T_\infty)\in [-\varepsilon, 0]\times {\R/\Z}$ that
$(\phi(s_k), t_k)\to (\phi_\infty, t_\infty)$
and
$(\phi(S_k), T_k)\to (\Phi_\infty, T_\infty)$.
Since $\beta_\sigma(\phi(S_k))>0$ it follows that
$\phi(S_k)\in (-2\varepsilon/(3\sigma_k), 0]$ so we can conclude that
$\Phi_\infty=0$ and thus that 
$a(\Phi_\infty, T_\infty)=a(0, T_\infty)=0$.
But since 
\begin{equation}\label{e:a-inquality}
a(\phi(S_k), T_k)<a(\phi(s_k), t_k)\le 0
\end{equation}
we can conclude that
$a(\phi_\infty, t_\infty)=0$ as well.
Since we have that
\[
u(\Phi_\infty, T_\infty)
=
\lim u(\phi(S_k), T_k)
=
\lim u(\phi(s_k), t_k)
=
u(\phi_\infty, t_\infty)
\]
we conclude that
\[
\tilde u(\Phi_\infty, T_\infty)
=
(0, u(\Phi_\infty, T_\infty))
=
(0, u(\phi_\infty, t_\infty))
=
\tilde u(\phi_\infty, t_\infty)
\]
which, since $\tilde u$ is an embedding implies that
\[
(\phi_\infty, t_\infty)=(\Phi_\infty, T_\infty)=(0, t_*)
\]
for some $t_*\in {\R/\Z}$.
Passing to a further subsequence if necessary, we can assume that the sequence of unit vectors
$v_k\in\R^2$ defined by
\[
v_k=\frac{(\phi(S_k)-\phi(s_k), T_k-t_k)}{\left\|(\phi(S_k)-\phi(s_k), T_k-t_k)\right\|},
\]
where we make sense of the right-hand side by choosing the unique representative of $T_k-t_k$ lying in $(-1/2, 1/2]$,
converges to some unit vector $v_\infty$.
Since $u(\phi(S_k), T_k)=u(\phi(s_k), t_k)$ for all $k$, we can conclude that
$du(0, t_*)(v_\infty)=0$.
Since $u_t(0, t_*)=\dot\gamma(t_*)\ne 0$, we can 
conclude that the $\partial_s$ component of
$v_\infty$ is nonzero and, since
$\phi(S_k)>\phi(s_k)$, positive.
It follows then from \eqref{e:a-inquality} that
$da(0, t_*)(v_\infty)\le 0$ which, since
$a_t(0, t)=0$ for all $t\in {\R/\Z}$, implies that
$a_s(0, t_*)\le 0$.
But since $a(0, t_*)=0$ and $a(s, t_*)\le 0$ for all $s\in [-\varepsilon, 0]$, we have that
$a_s(0, t_*)\ge 0$ as well, so we conclude that
$a_s(0, t_*)=0$ and thus $da(0, t_*)=0$.
We have thus found a point $(0, t_*)$ 
and a unit vector $v_\infty$ so that
\[
d\tilde u(0, t_*)(v_\infty)
=
(da(0, t_*)(v_\infty), du(0, t_*)(v_\infty))
=
0
\]
which contradicts the assumption that $\tilde u$ is an embedding.
This contradiction completes the proof that for some sufficiently large value of $\sigma$,
$\tilde v_\sigma$ is injective and therefore an embedding.
\end{proof}

\subsection{Concluding the proof of Theorem~\ref{main2}}

Recall that we may work under the assumption that $\lambda$ is nondegenerate.

Consider the projection $\Mfast_1(P,\jbar) \to \Mfast(P,\jbar)$ obtained by forgetting the marked point. This is a surjective submersion. Let $e$ be the homotopy class induced by the symplectic slicing disk $D$ for $P$. We define $\Mfast_1(P,\jbar,e)$ to be the pre-image of $\Mfast(P,\jbar,e)$ by the forgetful map above. By Lemma~\ref{lemma_existence} $\Mfast_1(P,\jbar,e)\neq\emptyset$. It was already observed before that 
\begin{equation}\label{ev_right_one}
\ev:\Mfast_1(P,\jbar,e) \to \W
\end{equation}
is a smooth submersion; see Remark~\ref{rmk_ev_submersion}. Hence its image is open and non-empty in $\W$. Lemma~\ref{lemma_int_invariance} implies that~\eqref{ev_right_one} is injective.

To prove surjectivity of~\eqref{ev_right_one} we will show that $\ev(\Mfast_1(P,\jbar,e))$ is closed in~$\W$. To this end suppose that $[\vtil_n,z_n] \in \Mfast_1(P,\jbar,e)$ and that $\ev[\vtil_n,z_n]$ converges to a point $q\in\overline{W}$. Up to choice of a subsequence, we may assume that $[\vtil_n,\C\cup\{\infty\},i,\{\infty\},\emptyset,\emptyset]$ SFT-converges to a building $\mathbf{u}$ of height $k_-|1|k_+$. We claim that $k_+=0$. If not then Lemma~\ref{lemma_comp_conseq_levels_above} implies that $V_0^+$ is not dynamically convex in $W$ relative to $(P,D)$. If $k_-\geq1$ then Lemma~\ref{lemma_comp_periods_lower_levels} provides a periodic Reeb orbit $P_*=(x_*,T_*)$ in $V^-$ such that $T_*$ satisfies~\eqref{action_ineq_lemma} for all $b>0$. The monotonicity lemma implies that the term $\int_{\util_0^{-1}(W)} \util_0^*\overline{\omega}_b$ has a positive lower bound independent of $b$, where $\util_0$ is the map representing the level zero of $\mathbf{u}$. Taking the limit as $b\to0^+$ we get $T_*<\int_D\omega$. This is one of the alternatives in Theorem~\ref{main2}, and we are now left with the case $k_-=0$. In this case, $\mathbf{u}$ consists of a single nodal curve $\util_0$. Lemma~\ref{lemma_comp_spheres_zero_level} tells us that there is a $\jbar$-holomorphic sphere if $\util_0$ has nodal points, but such spheres are excluded since $(W,\omega)$ is aspherical. Hence $\util_0$ is a finite-energy plane asymptotic to $P$. It is the $C^\infty_{\rm loc}$-limit of a suitable reparametrization of the planes in the sequence $\vtil_n$ which we still denote by $\vtil_n$. Lemma~\ref{lemma_limits_in_moduli_space} implies that $\util_0$ represents a curve in $\Mfast(P,\jbar,e)$. SFT convergence implies that $[\vtil_n,z_n]$ converges in $\Mfast_1(P,\jbar,e)$ to $[\util_0,\zeta]$ for some $\zeta\in\C$. Hence $\ev[\util_0,\zeta] = q$.

If $V^+_1$ is a connected component of $V^+$, $V^+_1\neq V^+_0$, then fix $q$ in $(0,+\infty) \times V^+_1$ and, using the above, find $[\vtil,z] \in \Mfast_1(P,\jbar,e)$ such that $\ev[\vtil,z]=q$. Hence we get $\jtil$-holomorphic disk with boundary on $\{0\}\times V^+_1$ through $q$, violating the maximum principle for the $\R$-component of this disk. It follows that $V^+ = V^+_0$.

Now take a sequence $q_n \in \W$ satisfying $q_n \to (+\infty,p)$ on $\W_\infty$, with $p\in V^+ \setminus x(\R)$. We claim that if $\ev[\vtil_n,z_n]=q_n$ then there is $n_0$ such that $n\geq n_0 \Rightarrow \vtil_n(\C) \subset [0,+\infty) \times V^+$. If not then we could apply Lemma~\ref{lemma_comp_forcing_levels_above} to conclude that the sequence $[\vtil_n,\C\cup\{\infty\},i,\{\infty\},\emptyset,\emptyset]$ SFT-converges to a building with $k_+>0$, and Lemma~\ref{lemma_comp_conseq_levels_above} contradicts the assumed dynamical convexity of $(V_0^+,\lambda)$ in $W$ relative to $(P,D)$. We now claim that if $n\geq n_0$ and $\vtil_n=(b_n,v_n)$ then $v_n$ is an immersion transverse to $X_\lambda$. This follows since $\vtil_n$, with $n\geq n_0$, defines a plane in $(\R\times V^+,\jtil)$ and we can estimate $0\leq \wind_\pi(\vtil_n) = \wind_\infty(\vtil_n)-1 \leq 0$. Moreover, $v_n$ is injective and does not intersect $x(\R)$. This is a consequence of $\vtil_n *_\delta (c\cdot \vtil_n) = 0$ for all $c>0$, where $c\cdot \vtil_n$ denotes $\R$-translation by $c>0$. Hence $v_n$ determines an embedded disk for $P$ in $V^+$ transverse to the Reeb vector field. It follows that $P$ is unknotted and has self-linking number $-1$. 

At this point it is simple to show that $(V^+,\ker\lambda) \simeq (S^3,\xi_0)$. Above we showed that there exists $[\vtil] \in \Mfast(P,\jbar,e)$ satisfying $\vtil(\C) \subset [0,+\infty)\times V^+$. Up to translating down we may assume that $\min b(\C) \in [0,1]$ where $\vtil=(b,v)$. The space of planes in $\Mfast(P,\jbar,e)$ that are contained in $[0,+\infty)\times V^+$ and whose images intersect $[0,1]\times V^+$ is then non-empty. It is also compact since a building obtained as the SFT-limit of a sequence of such planes has height $0|1|k_+$ and its zero level is contained in $[0,+\infty)\times V^+$, so Lemma~\ref{lemma_comp_conseq_levels_above} tells us that $k_+=0$, and the zero level has no spheres because $\jtil$ is compatible with an exact symplectic form on $[0,+\infty)\times V^+$. Hence the limiting building is a plane in $([0,+\infty)\times V^+,\jtil)$, which by Lemma~\ref{lemma_limits_in_moduli_space} belongs to $\Mfast(P,\jbar,e)$. As explained before, such a plane projects to an embedded spanning disk for $P$ in $V^+$ transverse to the flow. The compactness result just proved allows us to obtain an $S^1$-family of such spanning disks. Moreover, any two such planes $\vtil_j=(b_j,v_j)$, $j\in\{0,1\}$, satisfy either $v_0(\C)=v_1(\C)$ or $v_0(\C) \cap v_1(\C) =\emptyset$; in fact, if $v_0(\C) \cap v_1(\C) \neq \emptyset$ then, up to relabeling, we find $c\geq0$ such that $\vtil_0(\C) \cap (c\cdot \vtil_1)(\C) \neq \emptyset$, hence $c=0$ and $[\vtil_0]=[\vtil_1]$ by Lemma~\ref{lemma_intersections_fast_planes}, and $v_0(\C)=v_1(\C)$. It follows that the projections to $V^+$ of the planes in the $S^1$-family form an open book with disk-like pages supporting $(V^+,\ker\lambda)$. Thus $(V^+,\ker\lambda) \simeq (S^3,\xi_0)$. 

Finally we need to prove that $(W,\omega)$ is symplectomorphic to a star-shaped domain in $(\C^2,\omega_0)$. Since $(V=V^+,\ker\lambda)$ is contactomorphic to $(S^3,\xi_0)$, we can find a star-shaped domain $\Omega\subset\C^2$, $\epsilon>0$ small and a symplectomorphism $$\varphi : ([-\epsilon,+\infty)\times V^+,d(e^a\lambda)) \to ([-\epsilon,+\infty)\times \partial\Omega,d(e^a\lambda_0)). $$
Here $\lambda_0 = \iota^*\alpha_0$, where $\iota:\partial\Omega \to \C^2$ is the inclusion map and $\alpha_0$ is the standard Liouville form on $\C^2$. Note that $((-\frac{\epsilon}{2},+\infty)\times \partial\Omega,d(e^a\lambda_0))$ is symplectomorphic to $(\C^2\setminus K,\omega_0)$ for a compact set $K \subset \Omega \setminus\partial\Omega$. In fact, $K$ can be taken as a suitable scaling of $\Omega$. In view of the definition of the symplectic completion $(\overline{W},\overline{\omega})$, see~\S~\ref{ssec_symp_mfds}, $(-\frac{\epsilon}{2},+\infty)\times V^+$ can be seen as the complement in $\overline{W}$ of some compact subset of $W\setminus\partial W$. Hence, there is a suitable compact set $K' \subset W\setminus\partial W$ such that $\varphi$ can be used to define a symplectomorphism $(\overline{W}\setminus K',\overline{\omega}) \to (\C^2\setminus K,\omega_0)$. We can now apply Theorem~9.4.2 from~\cite{J_book} to conclude that $ (\overline{W},\overline{\omega})$ is symplectomorphic to $(\C^2,\omega_0)$ via a symplectomorphism that sends $W$ to $\Omega$.

\begin{remark}
Another interesting consequence from the above argument is that the supporting open book with binding $P$ in $(V^+,\xi)=(S^3,\xi_0)$ has pages that are global surfaces of section for the flow. This follows from arguments as in~\cite{fast}.
\end{remark}

\end{document}